\documentclass{article}
\usepackage[T1]{fontenc}
\usepackage[english]{babel}
\usepackage{textcomp,lmodern}
\usepackage{mathtools,amssymb,amsthm}
\usepackage{enumerate}
\newtheorem{lemma}{Lemma}[section]
\newtheorem{corollary}[lemma]{Corollary}
\newtheorem{proposition}[lemma]{Proposition}
\newtheorem{theorem}[lemma]{Theorem}
\theoremstyle{definition}
\newtheorem{definition}[lemma]{Definition}
\newtheorem{remark}[lemma]{Remark}
\newtheorem{notation}[lemma]{Notation}

\newtheorem{theoremA}{Theorem}[section]

\DeclareMathOperator\Aut{Aut}
\DeclareMathOperator\Rist{Rist}
\DeclareMathOperator\Stab{Stab}
\DeclareMathOperator\Sub{Sub}
\DeclarePairedDelimiter\abs{\lvert}{\rvert}
\DeclarePairedDelimiter\gen{\langle}{\rangle}
\newcommand*\defi[1]{\emph{#1}}
\newcommand*\GGS{\textrm{GGS}}
\newcommand*\Grig{{\mathcal G}}
\newcommand*\level[1]{\mathcal L_{#1}}
\newcommand*\setst[2]{\{#1\,|\,#2\}}
\newcommand*\subd{\leq_{\mathrm s}}
\newcommand*\restr[2]{#1_{\mkern 1mu \vrule height 2ex\mkern2mu #2}}
\newcommand*\Z{\mathbf{Z}}
\newcommand*\N{\mathbf{N}}
\newcommand*{\X}{\mathcal{X}}
\usepackage[colorlinks,breaklinks,bookmarks,plainpages=false,unicode=true]{hyperref}  
\newcommand*{\titre}{Subgroup induction property for branch groups}
\title{\titre}
\newcommand*{\auteur}{Dominik Francoeur, Paul-Henry Leemann}
\author{\auteur}
\newcommand*{\sujet}{branch groups, finitely generated subgroups, subgroup induction theorem, GGS groups}
\date{\today}
\hypersetup{pdftitle={\titre}, pdfauthor={\auteur}, pdfsubject={\sujet}}
\begin{document}
\maketitle
\begin{abstract}
The \defi{subgroup induction property} is a property of self-similar groups acting on rooted trees introduced by Grigorchuk and Wilson in 2003 that appears to have strong implications on the structure of the groups possessing it.
It was for example used in the proof that the first Grigorchuk group as well as the Gupta-Sidki $3$-group are subgroup separable (locally extended residually finite) or to describe their finitely generated subgroups as well as their weakly maximal subgroups.
However, until now, there were only two known examples of groups with this property, namely the first Grigorchuk group and the Gupta-Sidki $3$-group.

The aim of this article is twofold.
First, we investigate various consequences of the subgroup induction property for \defi{branch groups}, a particularly interesting class of self-similar groups.
Notably, we show that finitely generated branch groups with the subgroup induction property must be torsion, just infinite and subgroup separable, and we establish conditions under which all their maximal subgroups are of finite index and all their weakly maximal subgroups are closed in the profinite topology.
Then, we show that every torsion \GGS{} group has the subgroup induction property, hence providing the first infinite family of examples of groups with this property.
\end{abstract}

\section{Introduction}
Since their definition in 1997 (at the Groups St Andrews conference in Bath), branch groups (see Section~\ref{Section:Branch} for all the relevant definitions) have attracted a lot of attention.
Indeed, they are a rich source of examples of groups with exotic properties.
They also naturally appear in the classification of \defi{just infinite groups}~\cite{MR1765119}, that is infinite groups whose proper quotients are all finite.

One of the most well-known branch groups is the so-called first Grigorchuk group~$\Grig$, which was the first example of a group of intermediate growth~\cite{MR712546}.
Another well-studied family of branch groups is the family of \GGS{} groups, which consists of the Gupta-Sidki groups $G_p$ ($p$ prime) and their generalisations (see Section~\ref{Section:GGS}).
The first Grigorchuk group and the torsion \GGS{} groups have been intensively studied and share many properties: being just infinite, being finitely generated infinite torsion groups, having all maximal subgroups of finite index, and more.

In 2003, Grigorchuk and Wilson showed~\cite{MR2009443} that the Grigorchuk group $\Grig$ has a property now known as the \defi{subgroup induction property}.
Roughly speaking, a group with the subgroup induction property is a branch group such that any sufficiently nice property of the group is shared by all of its finitely generated subgroups (see Section~\ref{Section:General} for a proper definition).
They then used this fact to show that every finitely generated infinite subgroup of $\Grig$ is commensurable with~$\Grig$.
Recall that two groups $G_1$ and $G_2$ are \defi{commensurable} if there exists $H_i$ of finite index in $G_i$ with $H_1\cong H_2$.
In the same article, they also showed that $\Grig$ has the rather rare property of being subgroup separable, where $G$ is \defi{subgroup separable} (also called \defi{locally extended residually finite} or \defi{LERF}) if all of its finitely generated subgroups are closed in the profinite topology.
Among other things, this implies that $\Grig$ has a solvable membership problem, meaning that there exists an algorithm that, given $g\in G$ and a finitely generated subgroup $H\leq G$, decides whether $g$ belongs to $H$.
Indeed, this follows from subgroup separability and having an algorithm that lists all finite quotients, which is true both for $\Grig$ and \GGS{} groups.

Adapting the strategy of~\cite{MR2009443}, Garrido showed~\cite{MR3513107} that the Gupta-Sidki $3$-group $G_3$ also has the subgroup induction property, is subgroup separable and that any finitely generated  infinite subgroup $H\leq G_3$ is commensurable with~$G_3$ or with $G_3\times G_3$.

In recent years, the subgroup induction property of a branch group $G$ turned out to be a versatile tool.
For example, it was used in~\cite{MR4082048} to compute the Cantor-Bendixson rank of $\Grig$ and $G_3$, in~\cite{L2019} to describe the \defi{weakly maximal subgroups} (subgroups that are maximal for the property of being of infinite index) of $G$ and in~\cite{GLN2019} to give a characterisation of finitely generated subgroups of $G$.

The aim of this article is twofold: first, to expand the list of properties implied by the subgroup induction property, and second, to expand the list of groups known to satisfy this property.
With regard to the first goal, we consider two versions of the subgroup induction property, one a priori weaker than the other, see Subsections~\ref{SubSec:Def} and~\ref{SubSec:Alternative}.
With regard to the weak version, which we call the \defi{weak subgroup induction property}, we prove the following:
\begin{theoremA}\label{Thm:IntroWeak1}
Let $G$ be a finitely generated branch group with the weak subgroup induction property. Then $G$ is torsion and just infinite.
If moreover the intersection of all finite-index maximal subgroups of $G$ is not trivial (e.g. if $G$ is a $p$-group), then every maximal subgroup of $G$ is of finite index.
\end{theoremA}
If we moreover suppose that $G$ is self-replicating (see Definition~\ref{def:Self-Similar}), then we obtain
\begin{theoremA}\label{Thm:IntroWeak2}
Let $G\leq\Aut(T_d)$ be a finitely generated self-replicating branch group with the weak subgroup induction property and let $H$ be a finitely generated subgroup of~$G$.
Then, $H$ is commensurable with one of $\{1\}$, $G$, $G^{(\times 2)}$, \dots, $G^{(\times d-1)}$.
\end{theoremA}
Note that the article~\cite{FGLN2023} establishes conditions under which any finitely generated subgroup $H$ of $G$ possesses a so-called \defi{block structure}, from which it is easy to deduce that $H$ is commensurable with one of one of $\{1\}$, $G$, $G^{(\times 2)}$, \dots, $G^{(\times d-1)}$. However, our hypotheses for Theorem~\ref{Thm:IntroWeak2} are weaker than those of~\cite{FGLN2023}, so our result cannot be deduced from that one.

With regard to the stronger version of the subgroup induction property, which we simply call the \defi{subgroup induction property}, we show that it implies subgroup separability, among other things (see Theorems~\ref{thm:commmax} and~\ref{Thm:StrongSeparability} for more general statements).
\begin{theoremA}\label{Thm:IntroStrong}
Let $G\leq\Aut(T)$ be a finitely generated self-similar branch group with the subgroup induction property.
Then, $G$ is subgroup separable.

Suppose moreover the intersection of all finite-index maximal subgroups of $G$ is not trivial (e.g. if $G$ is a $p$-group), and let $H$ be a finitely generated subgroup of~$G$.
Then, all maximal subgroups of $H$ are of finite index in $H$ and all of its weakly maximal subgroups are closed in the profinite topology.
\end{theoremA}

Finally, the subgroup induction property also has consequences on the generalized membership problem (Corollary~\ref{Cor:GenWordProbl}), the Cantor-Bendixon rank of $G$ (Proposition~\ref{Prop:CBRank}), as well as on its cohomology (Proposition~\ref{Prop:Cohomology}).

In view of the above, it is natural to ask for examples of groups with the subgroup induction property.
Until now, the only known examples were the first Grigorchuk group $\Grig$ and the Gupta-Sidki 3-group $G_3$.
The second part of this article is devoted to providing infinitely many new examples of such groups.
\begin{theoremA}\label{Thm:Intro3}
All the torsion \GGS{} groups have the subgroup induction property.
\end{theoremA}
In fact, by Theorem~\ref{Thm:IntroWeak1} we have that a \GGS{} group has the subgroup induction property if and only if it is torsion.

\paragraph{Organisation}
The small Section~\ref{Section:Branch} quickly recalls the definitions of branch groups, self-similar groups and other related notions. The definition of the subgroup induction property and of the weak subgroup induction property are found in Subsection~\ref{SubSec:Def}.
Subsection~\ref{SubSec:wSIP} is about the consequences of the weak subgroup induction property and contains the proofs of Theorem~\ref{Thm:IntroWeak1}  and~\ref{Thm:IntroWeak2}.
The next subsection studies the subgroup induction property and proves Theorem~\ref{Thm:IntroStrong}.
The short Subsection~\ref{SubSec:Alternative} compares the relative strength of the weak subgroup induction property and of subgroup induction property.
The final section is dedicated to the proof of Theorem~\ref{Thm:Intro3}.

\paragraph{Acknowledgements}
The authors thank the anonymous reviewer for their careful reading and valuable remarks.
The authors performed this work within the framework of the LABEX MILYON (ANR-10-LABX-0070) of Universit\'e de Lyon, within the program ``Investissements d'Avenir'' (ANR-11-IDEX-0007) operated by the French National Research Agency (ANR).
The first named author was partly supported by the Leverhulme Trust Research Project Grant RPG-2022-025.
The second named author was partly supported by SNSF Grant No. $200021\_188578$ and by RDF-23-01-045-\emph{Groups acting on rooted trees and their subgroups} of XJTLU.
%
%
%
%
%
%
%
%
%
%
\section{Reminders on branch groups}\label{Section:Branch}
Let $d$ be an integer greater than $1$ and let $T$ be a $d$-regular rooted tree.
That is, $T$ is a tree with a special vertex, called the root, of degree $d$ and with all other vertices of degree $d+1$.
The set of all vertices at distance $n$ from the root is the $n$\textsuperscript{th} level of the tree and will be denoted by $\level{n}$.
There is a natural (partial) order on the vertices of $T$ defined by $w\geq v$ if the unique simple path starting from the root and ending at $w$ passes through $v$.
The root is the smallest element for this order.
We identify $T$ with the free monoid $\{1,\dots,d\}^*$ with the lexicographical order, where the children of $x_1\dots x_n$ are $x_1\dots x_ny$ for $y\in\{1,\dots,d\}$.
The set of infinite rays starting at the root is denoted by $\partial T$ and is isomorphic to the set of right-infinite words $\{1,\dots,d\}^\infty$.
There is a natural metrizable topology on $\partial T$ obtained by declaring that two rays are near if they share a long common prefix.
With this topology, the space $\partial T$ is compact.

A subset $X$ of $T$ is a \defi{transversal} (sometimes called a \defi{spanning leaf set}, a \defi{cut set} or a \defi{section}) if every ray of $\partial T$ passes through exactly one element of $X$.
It follows from the compactness of $\partial T$ that a transversal is always finite.

For any vertex $v$ of $T$, we denote by $T_v$ the subtree of $T$ consisting of all $w\geq v$.
It is naturally isomorphic\footnote{If we identify $T$ to $\{1,\dots,d\}^*$, then the identification of $T$ and $T_v$ is done by deletion of the prefix corresponding to $v$.} to $T$.

Let $G\leq \Aut(T)$ be a group of automorphisms of $T$ and let $V$ be a subset of vertices of $T$.
The subgroup $\Stab_G(V)=\bigcap_{v\in V}\Stab_G(v)$ is the pointwise stabilizer of $V$. 
The pointwise stabilizer $\Stab_G(\level{n})$ of $\level{n}$ is called the \defi{$n$\textsuperscript{th} level stabilizer}.
The \defi{rigid stabilizer} $\Rist_G(v)$ of a vertex $v$ in $G$ is the set of all elements of $G$ acting trivially outside~$T_v$.
That is, $\Rist_G(v)=\Stab_G(V)$ for $V=\setst{w\in T}{w\notin T_v}$.
Finally, the \defi{rigid stabilizer} of a level $\Rist_G(\level n)$ is the subgroup generated by the $\Rist_G(v)$ for all vertices $v$ of level $n$.
\begin{definition}
Let $T$ be a $d$-regular rooted tree.
A subgroup of $G\leq \Aut(T)$ is \defi{weakly branch} if $G$ acts transitively on the levels and all the rigid stabilizers of vertices are infinite.

If moreover all the $\Rist_G(\level n)$ have finite index in~$G$, then $G$ is said to be \defi{branch}.
\end{definition}
The class of branch groups naturally arises in the description of \defi{just infinite groups}, that is infinite groups whose quotients are all finite.
More precisely, a just infinite group is either a just infinite branch group, virtually a product of simple groups, or virtually a product of \defi{hereditary just infinite groups} (recall that a group $G$ is hereditary just infinite if all of its finite-index subgroups are just infinite), as was shown in~\cite{MR1765119}.
On the other hand, not all branch groups are just infinite.
It is thus interesting to investigate which branch groups are just infinite.

Apart from branch groups, another interesting class of subgroups of $\Aut(T)$ is the class of self-similar groups.
While these groups are of interest in their own right, they are often studied in conjunction with branch groups.
Indeed, the intersection of these two classes (that is groups that are branch and self-similar) contains a lot of interesting examples, including the first Grigorchuk group and the \GGS{} groups.
In order to introduce self-similar groups, we will need a little more notation.

Let $T$ be a $d$-regular rooted tree.
Recall that we identify the vertices of~$T$ with the free monoid $\{1,\dots, d\}^{*}$.
Thus, we have a natural operation of concatenation on the vertices of $T$.
Geometrically, given $v,w\in T$, the vertex $vw$ is the vertex of $T_v$ associated to $w$ under the natural identification between $T_v$ and $T$.

It follows from the definitions that for every automorphism $g\in \Aut(T)$ and every $v\in T$, there exists a unique automorphism $\restr{g}{T_v}\in \Aut(T)$ such that
\[g\cdot (vw) = (g\cdot v)(\restr{g}{T_v}\cdot w)\]
for all $w\in T$.
This gives us, for every $v\in T$, a map
\begin{align*}
	\varphi_v\colon \Aut(T)&\to\Aut(T)\\
	g&\mapsto \restr{g}{T_v}.
\end{align*}
This map restricts to a homomorphism between $\Stab(v)$ and $\Aut(T)$, and it restricts further to an isomorphism between $\Rist(v)$ and $\Aut(T)$, where we denote here by $\Stab(v)$ and $\Rist(v)$ the stabilizer and rigid stabilizer of $v$ with respect to the whole group $\Aut(T)$.
The element $\varphi_v(g)$ is called the \defi{section} of $g$ at $v$.
\begin{definition}\label{def:Self-Similar}
Let $T$ be a $d$-regular rooted tree. A group $G\leq \Aut(T)$ is said to be
\begin{enumerate}[(i)]
\item \defi{self-similar} if $\varphi_v(G)\leq G$ for all $v\in T$,
\item \defi{self-replicating} (or \defi{fractal}) if $\varphi_v\bigl(\Stab_G(v)\bigr)=G$ for all $v\in T$,
\item \defi{strongly self-replicating} (or \defi{strongly fractal}) if $\varphi_v\bigl(\Stab_G(\level 1)\bigr)=G$ for all $v\in \level{1}$
\item \defi{super strongly self-replicating} (or \defi{super strongly fractal}) if $\varphi_v\bigl(\Stab_G(\level n)\bigr)=G$ for all $n$ and all $v\in \level{n}$.\footnote{The prefix super is here to indicate that this property is true for all levels and not only for~$\level 1$. \GGS{} groups with constant vectors are strongly self-replicating but not super strongly self-replicating~\cite{Uria}.}%
\end{enumerate}
\end{definition}
A subgroup $H$ of $\prod_{i\in I}G_i$ is a \defi{subdirect product}, denoted $H\subd \prod_{i\in I}G_i$, if for every $i\in I$, the canonical projection $H\twoheadrightarrow G_i$ is surjective.
The notion of subdirect product plays an important role in the description of finitely generated subgroups of a group with the subgroup induction property, see~\cite{GLN2019,FGLN2023}.
%
%
%
%
%
%
%
%
%
%
%
%
\section{The subgroup induction property}\label{Section:General}
\subsection{Definitions and preleminaries}\label{SubSec:Def}
The first appearance (without a name) of the subgroup induction property dates back to~\cite{MR2009443}, where it was expressed in terms of inductive classes of subgroups.
\begin{definition}\label{Definition:StrongSubgroupInduction} Let $G\leq \Aut(T)$ be a self-similar group.
A class $\mathcal{X}$ of subgroups of $G$ is said to be \defi{inductive} if
\begin{enumerate}[(A)]
\item\label{Item:DefSubgroupInduction1}
Both $\{1\}$ and $G$ belong to $\mathcal{X}$,
\item\label{Item:DefSubgroupInduction2}If $H\leq L$ are two subgroups of $G$ with $[L:H]$ finite, then $L$ is in $\mathcal X$ if $H$ is in $\mathcal X$.
\item\label{Item:DefSubgroupInduction3}
If $H$ is a finitely generated subgroup of $\Stab_G(\level 1)$ and all first-level sections of $H$ are in $\mathcal{X}$, then $H\in \mathcal{X}$.
\end{enumerate}

The group $G$ has the \defi{subgroup induction property} if for any inductive class of subgroups $\mathcal X$, every finitely generated subgroup of $G$ is contained in $\mathcal X$.
\end{definition}

The first Grigorchuk group and the Gupta-Sidki $3$-group were shown in~\cite{MR2009443} and~\cite{MR3513107}, respectively, to possess the subgroup induction property. Up to now, these were the only groups for which this property was known to hold. Later in the present article, we will establish that it holds for all torsion \GGS{} groups (Theorem~\ref{thm:GGSHaveSubgroupInduction}).

The subgroup induction property has some fairly strong implications on the structure of a self-similar group $G$. In particular, it forces it to be strongly self-replicating, as the next proposition shows.
In the following we will use this fact without mentioning it.
\begin{proposition}\label{prop:SelfSimSIP}
Let $G\leq \Aut(T)$ be a finitely generated self-similar group with the subgroup induction property. Then, $G$ is strongly self-replicating.
\end{proposition}
\begin{proof}
Let $H\leq \Stab_G(\level 1)$ be any subgroup fixing all vertices on the first level. If there exists $v\in \level{1}$ such that $\varphi_v(H)= G$, then $G$ must be strongly self-replicating. Indeed, it follows from self-similarity that $\varphi_v(\Stab_G(\level 1))=G$, and by transitivity of the action of $G$ on $\level{1}$ that  $\varphi_w(\Stab_G(\level 1))=G$ for all $w\in \level{1}$.

Let us now suppose that $G$ is not strongly self-replicating, and let $\mathcal{X}$ be the collection of subgroups of $G$ consisting of $G$ and all of its finite subgroups. It is clear that $\mathcal{X}$ satisfies~\ref{Item:DefSubgroupInduction1} and~\ref{Item:DefSubgroupInduction2} of Definition~\ref{Definition:StrongSubgroupInduction}. If $H\leq \Stab_G(\level 1)$ is a finitely generated subgroup such that all its first-level sections are in $\mathcal{X}$, then all its first-level sections must be finite groups, since by the above discussion, no first-level section of $H$ can be equal to $G$. It follows that $H$ is finite and thus belongs to $\mathcal{X}$. Therefore, $\mathcal{X}$ is an inductive class, which is absurd since it does not contain all finitely generated subgroups of $G$.
Indeed, since $G$ is a finitely generated residually finite group, it contains finitely generated infinite proper subgroups.
\end{proof}

A slight modification of the Definition~\ref{Definition:StrongSubgroupInduction} leads to the a priori weaker notion of weak subgroup induction property, which in general does not seem to imply, at least not obviously, that the group must be strongly self-replicating. However, if the group $G$ satisfies some strong self-replicating assumption, then this weak version is equivalent to the full subgroup induction property; see Subsection~\ref{SubSec:Alternative}.
We say that a class $\mathcal{X}$ of subgroups of $G$ is \defi{strongly inductive} if it satisfies~\ref{Item:DefSubgroupInduction1},~\ref{Item:DefSubgroupInduction3} and 
\begin{enumerate}\renewcommand{\theenumi}{B'}
\renewcommand{\labelenumi}{(\theenumi)}
\item\label{Item:StrongInductive}
If $H\leq L$ are two subgroups of $G$ with $[L:H]$ finite, then $L$ is in $\mathcal X$ if and only if $H$ is in $\mathcal X$,
\end{enumerate}
\begin{definition}\label{Definition:SubgroupInduction}
The group $G$ has the \defi{weak subgroup induction property} if for any strongly inductive class of subgroups $\mathcal X$, every finitely generated subgroup of $G$ is contained in $\mathcal X$.
\end{definition}
Currently, there are no known examples of groups with the weak subgroup induction property but without the stronger version, and we do not know whether such groups can exist.
While the weak subgroup induction property will sometimes be enough for our purpose, the full strength of the subgroup induction property is sometimes needed to prove some results, see for example~\cite{FGLN2023} or Theorem~\ref{Thm:StrongSeparability}.

In~\cite{GLN2019}\footnote{Observe that in~\cite{GLN2019} what is called the ``subgroup induction property'' is the weak version.}, the second named author, together with R. Grigorchuk and T. Nagnibeda, proposed an alternative and more general definition of the weak subgroup induction property that makes sense for groups that are not self-similar, and that is sometimes easier to use. They proved~\cite[Proposition 4.3]{GLN2019} that for self-similar groups the two definitions are equivalent.
\begin{definition}[see~\cite{GLN2019}, Proposition 4.3]\label{Definition2}
A group $G\leq \Aut(T)$ has the \defi{weak subgroup induction property}
if for every finitely generated subgroup $H\leq G$, there exists a transversal $X$ of $T$ such that for each $v\in X$, the section $\varphi_v(\Stab_H(X))$ is either trivial or has finite index in $\varphi_v\bigl(\Stab_G(v)\bigr)$.
\end{definition}
Let $G\leq \Aut(T)$ and let $H$ be a subgroup of $G$. It follows from Definition~\ref{Definition2} that if $G$ has the weak subgroup induction property, so does $H$. Moreover, if $H$ has finite index in $G$ the converse is also true.
\subsection{Consequences of the weak subgroup induction property}\label{SubSec:wSIP}
We will now use Definition~\ref{Definition2} to prove several consequences of the weak subgroup induction property.
It was observed in~\cite[Corollary 4.5]{GLN2019} that Definition~\ref{Definition2} implies that a branch group with the weak subgroup induction property is either finitely generated or locally finite. We will in the following mostly focus our attention on the finitely generated case.

First of all, a finitely generated branch group with the weak subgroup induction property must be torsion and just infinite, hence proving the first part of Theorem~\ref{Thm:IntroWeak1}.
\begin{proposition}\label{Prop:JustInfinite}
Let $G\leq\Aut(T)$ be a finitely generated branch group with the weak subgroup induction property.
Then $G$ is torsion and just infinite.
\end{proposition}
\begin{proof}
By~\cite[Theorem 4]{MR1765119}, a finitely generated torsion branch group  is always just infinite.
Hence, the proposition directly follows from the next lemma.
\end{proof}
\begin{lemma}\label{Lemma:Torsion}
Let $G\leq\Aut(T)$ be a group with the weak subgroup induction property and such that $\Rist_G(v)$ is infinite for every $v\in T$\footnote{That is: the group $G$ is \defi{micro-supported}.}. 
Then, $G$ is torsion.
\end{lemma}
\begin{proof}
Let $g$ be any element of $G$ and let $H\coloneqq\gen{g}$.
Let $X$ be a transversal of $T$ such that for every $v\in X$, the section $\varphi_{v}\bigl(\Stab_H(X)\bigr)$ is either trivial or has finite index in $\varphi_{v}\bigl(\Stab_G(v)\bigr)$.
If $g$ is of infinite order, then at least one of these sections is infinite, which implies that for some $v\in X$, the subgroup $\varphi_{v}\bigl(\Stab_H(X)\bigr)$ is isomorphic to $\Z$ and hence $\varphi_{v}\bigl(\Stab_G(v)\bigr)$ is virtually~$\Z$.
But  then, the subgroup $\Rist_G(v)\cong\varphi_{v}\bigl(\Rist_G(v)\bigr)$ is also virtually $\Z$.
In particular, there exists a finite-index subgroup $Z$ of $\Rist_G(v)$ which is isomorphic to $\Z$.
Since $Z$ has finite index in $\Rist_G(v)$, for every $w\geq v$ the subgroup $\Rist_Z(w)\leq Z$ is non-trivial and hence also isomorphic to $\Z$.
But then, by looking at all children $w_1,\dots, w_d$ of $v$ we have that $Z\cong \Z$ contains a subgroup isomorphic to $Z^d\cong\Z^d$ which is absurd.
\end{proof}
Recall that the \defi{Frattini subgroup} $\Phi(G)$ of a group $G$ is the intersection of all the maximal subgroups of $G$.
In the following, we will make use of the bigger subgroup $\Phi_F(G)$ consisting of the intersection of maximal subgroups of finite index of $G$.
This is obviously a characteristic subgroup of $G$.
Observe that $\Phi_F(G)$ has finite index in $G$ as soon as $G$ is a finitely generated $p$-group.
More generally, we have
\begin{lemma}\label{Lemma:DerivedSubgroup}
Let $G$ be a group such that every finite quotient of $G$ is nilpotent.
Then $\Phi_F(G)$ contains the derived subgroup $G'$.
In particular, $\Phi_F(G)$ has finite index in $G$ if $G'$ does.
\end{lemma}
\begin{proof}
Let $M$ be a maximal subgroup of finite index and $N$ be its normal core.
Since $M$ is of finite index, so is $N$, and $M/N$ is a maximal subgroup of the finite nilpotent group $G/N$.
Since a finite group is nilpotent if and only if every maximal subgroup is normal, we conclude that $M$ is normal and hence $G/M$ is cyclic of prime order.
In particular, $G/M$ is abelian, which implies that $M$ contains $G'$.
\end{proof}
Observe that the condition that $G'$ has finite index in $G$ is satisfied for a broad class of groups.
This is the case for example for finitely generated torsion groups, or for non-abelian just infinite groups.
In particular, if $G$ is a finitely generated weakly branch (or more generally micro-supported) group with the weak subgroup induction property, then it is torsion and therefore $G'$ has finite index in $G$.
This implies that for such a $G$, the subgroup $\Phi_F(G)$ has finite index in $G$ as soon as every finite quotient of $G$ is nilpotent.
\begin{lemma}\label{Lemma:FiniteRankDense}
Let $G$ be a finitely generated group such that $\Phi_F(G)$ has finite index in $G$.
If $H$ is a subgroup of $G$ that is dense for the profinite topology, then it contains a finitely generated subgroup that is also dense in $G$ for the profinite topology.
\end{lemma}
\begin{proof}
Since $H$ is dense for the profinite topology and $\Phi_F(G)$ is of finite index, we have $H\cdot\Phi_F(G)=G$.
In particular, we have elements $h_1$ to $h_n$ of $H$ such that $\gen{h_1,\dots,h_n}\Phi_F(G)=G$.
Now, let $N$ be a normal subgroup of finite index of $G$.
We claim that $\gen{h_1,\dots,h_n}N=G$.
Indeed, if it was not the case, there would exist some maximal subgroup $M$ of $G$ containing $\gen{h_1,\dots,h_n}N$.
This subgroup is necessarily of finite index and hence contains $\Phi_F(G)$.
In particular, we would have $M\geq\gen{h_1,\dots,h_n}\Phi_F(G)=G$, which is absurd.
We hence have
\[
	\overline{\gen{h_1,\dots,h_n}}=\bigcap_{N\underset{\textnormal{f.i.}}{\trianglelefteq} G}\gen{h_1,\dots,h_n}N=\bigcap_{N\underset{\textnormal{f.i.}}{\trianglelefteq} G}G=G
\]
as desired.
\end{proof}
Using the above lemma, we prove that when $\Phi_F(G)$ is non-trivial, the subgroup induction property for a branch group implies that all maximal subgroups are of finite index.%
\begin{proposition}\label{Prop:MaxFinite}
Let $G\leq\Aut(T)$ be a finitely generated branch group with the weak subgroup induction property.
If $\Phi_F(G)$ is non-trivial, then every maximal subgroup of $G$ has finite index in $G$.
\end{proposition}
\begin{proof}
Suppose that $\Phi_F(G)$ is non-trivial.
Then, it must be of finite index in $G$, since $G$ is just infinite by Proposition \ref{Prop:JustInfinite}.

Suppose now for a contradiction that there exists a maximal subgroup $M$ that has infinite index in $G$.
Such a subgroup is dense for the profinite topology.
By Lemma~\ref{Lemma:FiniteRankDense}, $M$ contains a finitely generated proper dense subgroup $H<G$.

On one hand, since $G$ has the weak subgroup induction property, there must exist a transversal $X$ of $T$ such that for every $v\in X$, the section $\varphi_v(\Stab_H(X))$ is either trivial or has finite index in $\varphi_v(\Stab_G(v))$.
On the other hand, since $G$ is just infinite, being dense in the profinite topology on $G$ is equivalent to being pronormal (see Definition 2.21 of~\cite{Francoeur2020}).
Therefore, it follows from Lemma 3.1 and Theorem 3.2 of~\cite{Francoeur2020} that for every $v\in X$, the subgroup $\varphi_v(\Stab_H(v))$ is a proper and dense subgroup of $\varphi_v(\Stab_G(v))$ for the profinite topology on $\varphi_v(\Stab_G(v))$, since $H$ is proper and dense for the profinite topology on the branch group $G$.
This immediately leads to a contradiction.
Indeed, since $\varphi_v(\Stab_G(v))$ is a residually finite group, both finite subgroups and finite-index subgroups are closed for the profinite topology, and hence cannot be both dense and proper.
\end{proof}
This finishes the proof of Theorem~\ref{Thm:IntroWeak1}.

We now turn our attention to Theorem~\ref{Thm:IntroWeak2}. We start by introducing some notation and some preliminaries results.
\begin{notation}
Let $G$ be a group and $k\in \N$ be a natural number.
We will denote by $G^{(\times k)}$ the direct product of $k$ copies of $G$.
\end{notation}
\begin{lemma}\label{lemma:AlmostNormalSubgroupsCommensurableWithPower}
Let $G\leq\Aut(T_d)$ be a just infinite self-replicating branch group and let $H\leq G$ be an almost normal subgroup of $G$, i.e. a subgroup which is normal in a finite-index subgroup of $G$.
Then, $H$ is commensurable with one of $\{1\}$, $G$, $G^{(\times 2)}$, \dots, $G^{(\times d-1)}$.
\end{lemma}
\begin{proof}
If $H=1$, then the result is obvious.
If $H$ is non-trivial, then it follows from~\cite[Theorem 1.2]{GarridoWilson} that there exist $m,n\in \N$ with $m\geq n$ and some non-empty subset $V\subseteq \level{n}$ such that
\[\Rist_G(\level m)' \cap \Bigl(\prod_{v\in V}\Rist_G(v)\Bigr) \leq H \leq \prod_{v\in V}\Rist_G(v).\]
Since $G$ is just infinite, $\Rist_G(\level m)'$ is of finite index in $G$, from which we conclude that $H$ is commensurable with $\prod_{v\in V}\Rist_G(v)$.
It follows from the fact that $G$ is a self-replicating branch group that $\Rist_G(v)$ is commensurable with $G$ for every $v\in T_d$.
In particular, $H$ is commensurable with $G^{(\times |V|)}$, and $G$ is commensurable with $G^{(\times d)}$.
Thus, by choosing $j\in \{0,1,\dots, d-1\}$ such that $j\equiv |V| \mod d$, we have $H$ commensurable with $G^{(\times j)}$.
\end{proof}
\begin{lemma}\label{lemma:CommensurabilityOfSubdirectProduct}
Let $G\leq\Aut(T_d)$ be a just infinite self-replicating branch group and let $H_1,\dots, H_k$ be groups such that for every $i\in \{1,\dots, k\}$, $H_i$ is commensurable with one of $\{1\}$, $G$, $G^{(\times 2)}$, \dots, $G^{(\times d-1)}$.
Let $H\subd H_1\times \dots \times H_k$ be a subdirect product.
Then, $H$ is commensurable with one of $\{1\}$, $G$, $G^{(\times 2)}$, \dots, $G^{(\times d-1)}$.
\end{lemma}
\begin{proof}
In the case where $k=1$, there is nothing to prove.
Let us now suppose that $k=2$, so that $H\subd H_1\times H_2$.
For $i\in\{1,2\}$, let us denote by $\pi_i\colon H\rightarrow H_1$ the natural projection map on the coordinate $i$, and let $R_2\leq H_2$ be such that $1\times R_2=\ker \pi_1$.
Notice that $1\times R_2$ is a normal subgroup of $H$ and thus $R_2$ is a normal subgroup of $H_2$.

By our assumptions, there exists $j_2\in \{0,1,\dots, d-1\}$ such that $H_2$ is commensurable with $G^{(\times j_2)}$.
Using the fact that $G$ is a self-replicating branch group, and thus that $\Rist_G(v)$ is commensurable with $G$ for all $v\in T_d$, we can find a finite-index subgroup $L_2\leq H_2$, a (possibly empty) subset of vertices $V\subseteq \level{1}$ of the tree $T_d$ of cardinality $j_2$ and an injective group homomorphism $\iota\colon L_2 \rightarrow G$ such that $\iota(L_2)$ is a normal subgroup of finite index of $\prod_{v\in V}\Rist_G(v)$, which is commensurable with $G^{(\times j_2)}$.

Let us consider $\iota(L_2\cap R_2)$.
Since $R_2$ is normal in $H_2$, $\iota(L_2\cap R_2)$ must be normal in $\iota(L_2)$.
It follows that $\iota(L_2\cap R_2)$ is also normal in the direct product
\[
	\iota(L_2)\times\prod_{v\in \level{1}\setminus V}\Rist_G(v).
\]
Notice that this latter subgroup must be of finite index in $\Rist_G(\level{1})$, and therefore in $G$, since $\iota(L_2)$ is of finite index in $\prod_{v\in V}\Rist_G(v)$.
Thus, $\iota(L_2\cap R_2)$ is an almost normal subgroup of $G$ (recall that an \defi{almost normal} subgroup of $G$ is a subgroup that is normal in a subgroup of finite index of $G$).
In particular, by Lemma~\ref{lemma:AlmostNormalSubgroupsCommensurableWithPower}, there exists $j_{R}\in \{0,1,\dots, d-1\}$ such that $\iota(L_2\cap R_2)$ is commensurable with $G^{(\times j_{R})}$, which implies that $R_2$ is commensurable with $G^{(\times j_R)}$.
Furthermore, by Proposition 3.11 of~\cite{FGLN2023}, there exists an almost normal subgroup $K\leq G$ such that $K$ is a \defi{complement} of $\iota(L_2\cap R_2)$, meaning that $[K,\iota(L_2\cap R_2)]=1$, $K\cap \iota(L_2\cap R_2) = 1$ and $K\iota(L_2\cap R_2)$ is of finite index in $G$.

Let $K_2= \iota^{-1}(K\cap \iota(L_2))\leq L_2$.
Then, by the injectivity of the map $\iota$, we have $[K_2,L_2\cap R_2]=1$ and $K_2\cap (L_2\cap R_2)=1$.
Furthermore, since $K\iota(L_2\cap R_2)$ is of finite index in $G$, $(K\cap \iota(L_2))\iota(L_2\cap R_2)$ must be of finite index in $\iota(L_2)$, which implies that $K_2(L_2\cap R_2)$ is of finite index in $L_2$ and thus in $H_2$.

Let us define $P=\pi_2^{-1}(K_2)$.
Then, we must have the following properties
\begin{enumerate}
\item $[P,1\times (L_2\cap R_2)]=1$,\label{item:PAndRCommute}
\item $P\cap (1\times (L_2\cap R_2))=1$,\label{item:PAndRDisjoint}
\item $P(1\times (L_2\cap R_2))$ is of finite index in $H$,\label{item:PRFiniteIndex}
\end{enumerate}
where point~\ref{item:PRFiniteIndex} follows from the fact that $\pi_2(P(1\times (L_2\cap R_2))) = K_2(L_2\cap R_2)$ is of finite index in $L_2$, which is of finite index in $H_2$, and that the kernel of $\pi_2$ is contained in $P$.

It follows from points~\ref{item:PAndRCommute}, \ref{item:PAndRDisjoint} and~\ref{item:PRFiniteIndex} that $H$ is abstractly commensurable with $P\times R_2$.
Furthermore, we see from point~\ref{item:PAndRDisjoint} and the fact that $\ker(\pi_1)\cap \pi_2^{-1}(L_2) = 1\times (L_2\cap R_2)$ that $\pi_1$ is injective when restricted to $P$.
Therefore, $P\cong \pi_1(P)$, which is of finite index in $H_1$ by point~\ref{item:PRFiniteIndex}.
Putting all this together, we conclude that $H$ is abstractly commensurable with $H_1\times R_2$.
We have seen above that $R_2$ is commensurable with $G^{(\times j_R)}$, and by our assumptions on $H_1$, there exists $j_1\in \{0,1,\dots, d-1\}$ such that $H_1$ is commensurable with $G^{(\times j_1)}$.
Thus, $H$ is commensurable with $G^{(\times (j_1+j_R))}$.
Since $G$ is a self-replicating branch group, it must be abstractly commensurable with $G^{(\times d)}$, and therefore $G^{(\times (j_1+j_R))}$ is either trivial (when $j_1=j_R=0$) or commensurable with $G^{(\times j_0)}$, where $j_0$ is the unique element in $\{1,\dots ,d-1\}$ such that $j_0 \equiv j_1+j_R \mod d-1$.

For the cases where $k>2$, we proceed by induction.
Let us assume that the result holds for some $k$, and let us show it for $k+1$.
In this case, we have $H\leq H_1 \times H_2\times \dots \times H_{k+1}$ a subdirect product.
Let us denote by
\[
	\pi_{\perp}\colon H_1\times H_2\times \dots \times H_{k+1} \rightarrow H_2\times \dots \times H_{k+1}
\]
the homomorphism obtained by deleting the first coordinate, and let us write $H^{\perp} = \pi_{\perp}(H)$.
Notice that $H^{\perp}\leq H_2\times \dots \times H_{k+1}$ is a subdirect product.
Therefore, by our induction hypothesis, $H^{\perp}$ is commensurable with one of $\{1\}$, $G$, $G^{(\times 2)}$, \dots, $G^{(\times d-1)}$.
This means that $H\leq H_1\times H^{\perp}$ is a subdirect product of two groups satisfying the assumptions of the theorem.
As we have shown above, this implies that $H$ is commensurable with one of $\{1\}$, $G$, $G^{(\times 2)}$, \dots, $G^{(\times d-1)}$.
\end{proof}
\begin{theorem}\label{Thm:Intro2Generalized}
Let $G\leq\Aut(T_d)$ be a finitely generated self-replicating branch group with the weak subgroup induction property and let $H$ be a finitely generated subgroup of~$G$.
Then, $H$ is commensurable with one of $\{1\}$, $G$, $G^{(\times 2)}$, \dots, $G^{(\times d-1)}$.
\end{theorem}
\begin{proof}
Let $\X$ be the set of all subgroups of $G$ which are commensurable with one of $\{1\}$, $G$, $G^{(\times 2)}$, \dots, $G^{(\times d-1)}$.
We need to show that $\X$ is a strongly inductive class.

It is clear that both $\{1\}$ and $G$ belong to $\X$, so condition (\ref{Item:DefSubgroupInduction1}) is satisfied.
Condition (\ref{Item:StrongInductive}) is also obviously satisfied, since if $H,L\leq G$ are two subgroups with $H$ of finite index in $G$, then they are commensurable, and commensurability is an equivalence relation.

The fact that condition (\ref{Item:DefSubgroupInduction3}) is also verified follows directly from Lemma~\ref{lemma:CommensurabilityOfSubdirectProduct} and the fact that $G$ must be just-infinite (Proposition~\ref{Prop:JustInfinite}). 
\end{proof}
\subsection{Consequences of the subgroup induction property}\label{SubSec:SIP}
By assuming that our group has the subgroup induction property instead of the weak subgroup induction property, one can extend Proposition~\ref{Prop:MaxFinite} to all groups commensurable with it.
We prove this immediately after recalling the following result from~\cite{MR2009443}. 
\begin{lemma}[{\cite[Lemma 1]{MR2009443}}]\label{Lemma:MaxInSubgroup}
Let $G$ be an infinite finitely generated group and let $H$ be a subgroup of finite index in $G$.
If $G$ has a maximal subgroup $M$ of infinite index, then $H$ has a maximal subgroup of infinite index containing $H \cap M$.
\end{lemma}
\begin{proposition}\label{prop:CommensurableMaxSubgroups}
Let $G\leq\Aut(T_d)$ be a finitely generated self-similar branch group with the subgroup induction property and such that $\Phi_F(G)$ is non-trivial, and let $\Gamma$ be a group commensurable with $G$.
Then, all maximal subgroups of $\Gamma$ are of finite index in $\Gamma$.
\end{proposition}
\begin{proof}
By Lemma~\ref{Lemma:MaxInSubgroup}, it suffices to show that if $H\leq G$ is a subgroup of finite index in $G$, then all of its maximal subgroups are of finite index.
Let us suppose that there exists a subgroup $H\leq G$ of finite index containing a maximal subgroup $M<H$ of infinite index.
As $M$ is maximal in $H$, which is of finite index in $G$, we must have $\overline{M}=H$, where $\overline{M}$ denotes the closure of $M$ in the profinite topology on $G$.

By~\cite[Proposition 5.15]{FGLN2023}, since $H$ is of finite index in $G$, there exists some $n\in \N$ such that $\varphi_v(\Stab_{H}(v)) = G$ for all $v\in \level{n}$.
For every $v\in \level{n}$, we observe that $\overline{\Stab_M(v)} = \Stab_H(v)$.
Indeed, for all $m\geq n$, we have
\[(M\Rist_G(\level m)')\cap \Stab_G(v)= \Stab_M(v)\Rist_G(\level m)',\]
since $\Rist_G(\level m)'\leq \Stab_G(v)$. It follows that
\begin{align*}
\Stab_{H}(v) &= \overline{M}\cap \Stab_G(v) =  \Bigl(\bigcap_{m\geq n}M\Rist_G(\level m)'\Bigr)\cap\Stab_G(v)\\
&= \bigcap_{m\geq n} (M\Rist_G(\level m)')\cap \Stab_G(v)\\
&=\bigcap_{m\geq n} \Stab_M(v)\Rist_G(\level m)'\\
&=\overline{\Stab_M(v)}
\end{align*}
where we used the fact that $G$ is a just infinite branch group, so that the collection $\setst{\Rist_G(\level m)'}{m\geq n}$ forms a basis of neighbourhoods for the profinite topology on $G$.

Since the preimage by $\varphi_v$ of a finite-index subgroup is a finite-index subgroup, this map is continuous for the profinite topology.
Hence, for every $v\in \level{n}$ we have
\[
	G = \varphi_v(\Stab_H(v)) = \varphi_v(\overline{\Stab_M(v)}) \leq \overline{\varphi_v(\Stab_M(v))}\leq G,
\]
from which it follows that $\varphi_v(\Stab_M(v))$ is dense in the profinite topology on $G$.
Since $G$ cannot contain maximal subgroups of infinite index by Proposition~\ref{Prop:MaxFinite}, this implies that $\varphi_v(\Stab_M(v))=G$ for all $v\in \level{n}$.

However, since $G$ is just infinite by Proposition~\ref{Prop:JustInfinite}, the pro-normal closure (see Definition 3.1 of~\cite{Francoeur2021}) of $M$ in $G$ is the same as its closure in the profinite topology, which is of finite index.
Thus, $M$ is a quasi-prodense subgroup of infinite index in $G$, and therefore, by Theorem 3.4 of~\cite{Francoeur2021}, there must exist some $v\in \level{n}$ such that $\varphi_v(\Stab_M(v))$ is of infinite index in $G$. 
This contradicts the conclusion of the previous paragraph, and we deduce that such a subgroup $M$ cannot exist.
\end{proof}
Using Theorem~\ref{Thm:Intro2Generalized}, we can extend Proposition~\ref{prop:CommensurableMaxSubgroups} to all groups commensurable with a finitely generated subgroup of $G$.
\begin{theorem}\label{thm:commmax}
Let $G\leq\Aut(T_d)$ be a finitely generated self-similar branch group with the subgroup induction property and such that $\Phi_F(G)$ is non-trivial.
Let $H$ be a group commensurable with a finitely generated subgroup of~$G$.
Then, all maximal subgroups of $H$ are of finite index in $H$.
\end{theorem}
\begin{proof}
By Theorem~\ref{Thm:Intro2Generalized}, $H$ is commensurable with $G^{(\times k)}$ for some $0\leq k \leq d-1$.
Therefore, $H\times G^{(\times d-k)}$ is commensurable with $G^{(\times d)}$ and thus with $G$.
If $M\leq H$ is a maximal subgroup of infinite index of $H$, then $M\times G^{(\times d-k)}$ is a maximal subgroup of infinite index of a subgroup commensurable with $G$, which contradicts Proposition~\ref{prop:CommensurableMaxSubgroups}.
Thus, all maximal subgroups of $H$ must be of finite index.
\end{proof}
As a consequence of Theorem~\ref{thm:commmax}, we obtain the following fact regarding weakly maximal subgroups (recall that a subgroup is weakly maximal if it of infinite index and maximal for this property), which finishes the proof of the second part of Theorem~\ref{Thm:IntroStrong}.
\begin{corollary}\label{cor:wmax}
Let $G\leq\Aut(T_d)$ be a finitely generated self-similar branch group with the subgroup induction property and such that $\Phi_F(G)$ is non-trivial.
Let $H$ be a group commensurable with a finitely generated subgroup of~$G$.
Then, all weakly maximal subgroups of $H$ are closed in the profinite topology.
\end{corollary}
\begin{proof}
Let $W$ be a weakly maximal subgroup of $H$ and suppose that $W\ne \overline{W}$, where $\overline{W}$ denote the closure of $W$ in the profinite topology. By weak maximality, every subgroup strictly containing $W$ is of finite index in $H$ and thus contains $\overline{W}$. Therefore, $W$ is maximal in $\overline{W}$, with $\overline{W}$ of finite index in $H$ and $W$ of infinite index in $H$ and therefore in $\overline{W}$. This contradicts Theorem~\ref{thm:commmax}, since no group commensurable with $H$ can contain a maximal subgroup of infinite index.
\end{proof}
In order to prove the first part of Theorem~\ref{Thm:IntroStrong}, we need the following technical lemma.
\begin{lemma}\label{Lemma:FromProjectionsToDirectProduct}
Let $G_1,\dots, G_n$ be groups, and let $H\leq G_1\times \dots \times G_n$ be a subgroup of their direct product.
For $1\leq i \leq n$, let us denote by $\pi_i\colon G_1\times \dots \times G_n\to G_i$ the natural projection.
If, for all $1\leq i \leq n$, every normal subgroup of $\pi_i(H)$ is finitely generated and closed in $G_i$ with respect to the profinite topology on $G_i$, then every normal subgroup of $H$ is finitely generated and closed in the profinite topology on $G_1\times \dots \times G_n$.
\end{lemma}
\begin{proof}
Let us first prove the result in the case where $n=2$.
Suppose that $H\leq G_1\times G_2$ is such that for all $i\in \{1,2\}$, every normal subgroup of $\pi_i(H)$ is finitely generated and closed in $G_i$ with respect to the profinite topology on $G_i$ (note that this implies in particular that both $G_1$ and $G_2$ must be residually finite) and let $N\trianglelefteq H$ be a normal subgroup of $H$.
We want to show that $N$ is finitely generated and closed in $G_1\times G_2$, with respect to the profinite topology.

For $i=1,2$, let us denote $N_i=\pi_i(N)$.
As $N_i$ is normal in $\pi_i(H)$, it follows from our assumptions that $N_i$ is finitely generated and closed in the profinite topology on $G_i$.
Let $M_1\leq N_1$ be the largest subgroup such that $M_1\times \{1\}\leq N$, and similarly, let $M_2\leq N_2$ be the largest subgroup such that $\{1\}\times M_2\leq N$.
It follows from their definition and the fact that $N$ is normal in $H$ that $M_i$ is normal in $\pi_i(H)$, and thus finitely generated and closed in the profinite topology on $G_i$, for $i=1,2$.

By Goursat's lemma, there exists an isomorphism $\alpha\colon N_1/M_1\rightarrow N_2/M_2$ such that
\begin{equation}
	N/(M_1\times M_2) = \setst{(g,\alpha(g))\in N_1/M_1\times N_2/M_2}{g\in N_1/M_1}.\label{eq:DiagonalQuotientGoursat}
\end{equation}
In particular, $N/(M_1\times M_2)$ is isomorphic to $N_1/M_1$, which is finitely generated, since $N_1$ is.
As $M_1\times M_2$ is also finitely generated, we conclude that $N$ must be so as well, being an extension of a finitely generated subgroup by another.
Let us now focus on showing that $N$ is closed in the profinite topology on $G_1\times G_2$.

For $i=1,2$, let us equip $N_i$ with the subspace topology inherited from the profinite topology on $G_i$, and let us subsequently equip $N_i/M_i$ with the quotient topology.
With this topology, both $N_1/M_1$ and $N_2/M_2$ are topological groups, and we claim that the isomorphism $\alpha$ is a homeomorphism.
In order to prove this, let us first notice that for $i=1,2$, every finite-index subgroup of $N_i$ is open.
Indeed, if $L\leq N_i$ is any finite-index subgroup, then its characteristic core $C\leq L$ is also of finite index in $N_i$, since $N_i$ is finitely generated.
Having $C$ characteristic in $N_i\trianglelefteq \pi_i(H)$, we conclude that $C$ is normal in $\pi_i(H)$ and thus closed in $G_i$ by assumption.
Since the topology on $N_i$ is the subspace topology, $C$ is also closed in $N_i$.
However, being a closed subgroup of finite index in a topological group, $C$ must also be open in $N_i$, and as we have $C\leq L$, we conclude that $L$ must also be open in $N_i$, which finishes showing that every finite-index subgroup of $N_i$ is open and hence clopen.

The next ingredient we need in our proof that $\alpha$ is a homeomorphism is that subgroups of finite index in $N_i/M_i$ form a basis of neighbourhoods of the identity.
This follows from the fact that finite-index subgroups of $G_i$ form a basis of neighbourhoods of the identity for the profinite topology on $G_i$.
Indeed, if $U\subseteq N_i/M_i$ is an open set containing the identity, then by the definition of the topologies, there exists an open set $V\subseteq G_i$ containing the identity such that $V\cap N_i$ is the preimage of $U$ under the projection map $N_i\rightarrow N_i/M_i$.
Then, there is a finite-index subgroup $L\leq G_i$ such that $L\subseteq V$.
Therefore, $L\cap N_i$ is a finite-index subgroup of $N_i$ contained in $V\cap N_i$, and thus $LM_i/M_i$ is a finite-index subgroup of $N_i/M_i$ contained in $U$.

As $\alpha\colon N_1/M_1 \rightarrow N_2/M_2$ is an isomorphism, it sends finite-index subgroups of $N_1/M_1$ to finite-index subgroups of $N_2/M_2$.
Since we have just seen that finite-index subgroups of $N_i/M_i$ form a basis of neighbourhoods of the identity, we can conclude, using the fact that we are in a topological group, that $\alpha$ sends open sets to open sets.
As the same argument also works with $\alpha^{-1}$, we conclude that $\alpha$ is a homeomorphism.

Having established this, let us come back to our task of showing that $N$ is closed in $G_1\times G_2$.
To do this, it suffices to show that $N/(M_1\times M_2)$ is closed in the group $G_1/M_1 \times G_2/M_2$ equipped with the quotient topology.
Indeed, if we denote by $\pi\colon G_1\times G_2 \rightarrow G_1/M_1 \times G_2/M_2$ the quotient map, which is continuous by definition, then $N=\pi^{-1}(N/(M_1\times M_2))$, so it must be closed.
Furthermore, since $N_i$ is closed in $G_i$ and contains $M_i$ for $i=1,2$, the subgroup $N_1/M_1 \times N_2/M_2$ is closed in $G_1/M_1\times G_2/M_2$.
Thus, it suffices to prove that $N$ is closed in the group $N_1/M_1 \times N_2/M_2$ equipped with the subspace topology (which is the same as the product topology coming from the previously considered topologies on $N_1/M_1$ and $N_2/M_2$).
By equation (\ref{eq:DiagonalQuotientGoursat}), we see that $N/(M_1\times M_2)$ is nothing else than the graph of the continuous map $\alpha\colon N_1/M_1 \rightarrow N_2/M_2$. Since $M_2$ is closed in $G_2$, the group $N_2/M_2$ is Hausdorff, and since the graph of a continous map into a Hausdorff space is closed, we conclude that $N/(M_1\times M_2)$ is closed in $N_1/M_1 \times N_2/M_2$ and thus that $N$ is closed in $G_1\times G_2$ by the above discussion.
This finishes the proof for the case $n=2$.

To treat the case $n>2$, we proceed by induction.
Let us suppose that the result holds for some $n\geq 2$, and let $H\leq G_1\times \dots \times G_{n+1}$ be a subgroup of a direct product of $n+1$ groups satisfying the assumptions of the lemma.
We want to show that every normal subgroup of $H$ is finitely generated and closed in the profinite topology on $G_1\times \dots \times G_{n+1}$.

Let us write $G_1^{\perp} = G_2\times \dots \times G_{n+1}$, so that $H\leq G_1\times G_1^{\perp}$, and let us denote by $\pi_{1^{\perp}}\colon G_1\times G_1^{\perp}\rightarrow G_1^{\perp}$ the natural projection on this factor.
Then, $\pi_{1^{\perp}}(H)$ is a subgroup of $G_2\times \dots \times G_{n+1}$ such that for every $i\in \{2,\dots, n+1\}$, every normal subgroup of $\pi_i(\pi_{1^{\perp}}(H))$ is finitely generated and closed in the profinite topology on $G_i$.
Indeed, this follows directly from our assumptions on $H$ and the fact that $\pi_i(\pi_{1^{\perp}}(H)) = \pi_i(H)$.
Therefore, by our induction hypothesis, every normal subgroup of $\pi_{1^{\perp}}(H)$ is finitely generated and closed in the profinite topology on $G_1^{\perp} = G_2\times \dots \times G_{n+1}$.
Therefore, we can apply the case $n=2$ to $H\leq G_1\times G_1^{\perp}$ to conclude that every normal subgroup of $H$ is finitely generated and closed in the profinite topology on $G_1\times G_1^{\perp} = G_1\times \dots \times G_{n+1}$.
\end{proof}
Using the previous lemma, we can prove subgroup separability for self-similar branch groups with the subgroup induction property, thus finishing the proof of Theorem~\ref{Thm:IntroStrong}.
\begin{theorem}\label{Thm:StrongSeparability}
Let $G\leq\Aut(T)$ be a finitely generated self-similar branch group with the subgroup induction property.
Then, for each finitely generated subgroup $H\leq G$ and each normal subgroup $N\trianglelefteq H$, $N$ is finitely generated and closed in the profinite topology on $G$.
In particular, $G$ is subgroup separable.
\end{theorem}
\begin{proof}
Let $\X$ be the class of subgroups of $G$ such that all their normal subgroups are finitely generated and closed in the profinite topology on $G$.
As $G$ has the subgroup induction property, it suffices to prove that this class is inductive.

Firstly, since $G$ is residually finite, we have $\{1\}\in \X$, and it follows from the fact that $G$ is just infinite (Theorem~\ref{Thm:IntroWeak1}) and finitely generated that $G\in \X$.
Thus, $\X$ satisfies Condition~\eqref{Item:DefSubgroupInduction1} of an inductive class.

Secondly, let $H\leq L$ be two subgroups of $G$, with $H$ of finite index in $L$.
Let us suppose that $H\in \X$, and let $N\trianglelefteq L$ be any normal subgroup of $L$.
Then, $N\cap H$ is a normal subgroup of $H$ and is thus finitely generated and closed in the profinite topology on $G$.
As $N\cap H$ is of finite index in $N$, the same then holds for $N$.
This shows that $L\in \X$ and thus that $\X$ satisfies Condition~\eqref{Item:DefSubgroupInduction2} of an inductive class.

Lastly, let $H$ be a subgroup of $\Stab_G(\level 1)$ such that all first-level sections of $H$ are in $\X$.
Let us denote by $\psi\colon \Stab_G(\level{1}) \rightarrow G^{\level{1}}$ the group homomorphism defined by $\psi(g) = (\varphi_v(g))_{v\in \level{1}}$.
It is immediate that $\psi$ is injective.
We have $\psi(H)\leq G^{\level{1}}$, with $\pi_v(\psi(H))=\varphi_v(H)\in \X$ for all $v\in \level{1}$, where $\pi_v\colon G^{\level{1}}\rightarrow G$ is the natural projection onto the component indexed by $v$.
Thus, by Lemma~\ref{Lemma:FromProjectionsToDirectProduct}, every normal subgroup of $\psi(H)$ is finitely generated and closed in the profinite topology on $G^{\level{1}}$.
Since $\psi$ is injective we immediately conclude that every normal subgroup of $H$ is finitely generated.
As $\psi$ is a group homomorphism, it is continuous with respect to the profinite topologies on $\Stab_G(\level{1})$ and~$G^{\level{1}}$.
Therefore, every normal subgroup of $H$ is closed in the profinite topology on $\Stab_G(\level 1)$, since it is the preimage by $\psi$ of a closed subgroup of $G^{\level{1}}$, where we used once again the fact that $\psi$ is injective. 
Since $\Stab_G(\level{1})$ is of finite index in $G$, any closed subgroup in the profinite topology on $\Stab_G(\level{1})$ must also be closed in the profinite topology on $G$.
Indeed, recall that a subgroup is closed in the profinite topology if and only if it is the intersection of finite-index subgroups, and finite-index subgroups of $\Stab_G(\level{1})$ are also finite-index subgroups of $G$.
Thus, we have verified that $\X$ also satisfies Condition~\eqref{Item:DefSubgroupInduction3} of an inductive class, which concludes the proof.
\end{proof}
As a direct corollary, we obtain that many groups with the subgroup induction property have a solvable membership word problem, meaning that there exists an algorithm that, given $g\in G$ and a finitely generated subgroup $H\leq G$, decides whether $g$ belongs to $H$.
Before stating the result, we recall that a group $G\leq \Aut(T)$ has the \defi{congruence subgroup property} if any finite-index subgroup contains the stabilizer $\Stab_G(\level n)$ of a level.
\begin{corollary}\label{Cor:GenWordProbl}
Let $G\leq\Aut(T)$ be a finitely generated self-similar branch group with both the congruence subgroup property and the subgroup induction property.
Then $G$ has a solvable membership problem.
\end{corollary}
\begin{proof}
For a finitely generated subgroup $G$ of $\Aut(T)$, we have an algorithm enumerating all the stabilizer $\Stab_G(\level n)$ of a level. If $G$ has the congruence subgroup property, this gives us an algorithm enumerating all the finite-index subgroups and hence all the finite quotients of $G$.
Together with subgroup separability, this implies that $G$ has a solvable membership problem \cite{Malcev}.
\end{proof}
As another application of the subgroup induction property, we expand a result of Skipper and Wesolek. See~\cite{MR4082048} for more and details on the Cantor–Bendixson rank.
\begin{proposition}\label{Prop:CBRank}
Let $G\leq \Aut(T)$ be a finitely generated super strongly self-replicating branch group with both the subgroup induction property and the congruence subgroup property.
Then $\Sub(G)$ has Cantor–Bendixson rank~$\omega$.
\end{proposition}
\begin{proof}
The group $G$ is just infinite by Proposition~\ref{Prop:JustInfinite} and subgroup separable by Theorem~\ref{Thm:StrongSeparability}.
Since it is subgroup separable and has the congruence subgroup property, it has ``well-approximated subgroups''~\cite[Lemma 2.12]{MR4082048}.
Moreover, the so-called ``Grigorchuk-Nagnibeda alternative''~\cite[Definition 4.1]{MR4082048} is a weak form of the weak subgroup induction property in the sense of Definition~\ref{Definition2}.
Hence~\cite[Corollary 5.4]{MR4082048} applies.
\end{proof}
It was brought to our attention by Anitha Thillaisundaram that as a consequence of Theorem~\ref{Thm:Intro2Generalized}, we can extend a result of Gandini on the rational cohomology of the Grigorchuk group~\cite{MR2996406} to a larger class of groups. This yields in particular more counterexamples to a conjecture of Petrosyan~\cite{MR2324601} as well as to a question of Jo and Nucinkis~\cite{MR2405893}. We refer the interested reader to~\cite{MR2996406} for the relevant definitions and details.
\begin{proposition}\label{Prop:Cohomology}
Let $G\leq\Aut(T_d)$ be a finitely generated self-replicating branch group with the weak subgroup induction.
Then, $G$ has jump rational cohomology of height~$1$ and infinite rational cohomology.
\end{proposition}
\begin{proof}
This follows from Theorem~\ref{Thm:Intro2Generalized}, Theorem 4.9 of~\cite{MR2996406} and Remark 4.6 of~\cite{MR2996406}.
\end{proof}
\subsection{Weak subgroup induction property versus subgroup induction property.}\label{SubSec:Alternative}
While we were able to prove some results of Section~\ref{Section:General} under the hypothesis that $G$ has the weak subgroup induction property, for all the results of Subsection~\ref{SubSec:SIP} we require that $G$ has the stronger version of the subgroup induction property.

It directly follows from the definitions that the subgroup induction property implies the weak subgroup induction property and it is natural to ask if the converse is true.
Observe that, by Proposition~\ref{prop:SelfSimSIP}, a self-similar subgroup with the subgroup induction property is automatically strongly self-replicating. It hence seems reasonable to restrict our study to the class of strongly self-replicating group. While we are not yet able to prove that the two versions of the subgroup induction property are equivalent for strongly self-replicating groups we can show it under a slightly stronger hypothesis.
\begin{proposition}
Let $G$ be a super strongly self-replicating group with the weak subgroup induction property. Suppose that moreover $G$ has the congruence subgroup property. Then $G$ has the subgroup induction property.
\end{proposition}
This proposition follows from Lemma~\ref{Lemma:fisdp}.
Indeed, it is a simple verification that any super strongly self-replicating group $G$ with the congruence subgroup property satisfies the hypothesis of Lemma~\ref{Lemma:fisdp} on finite-index subgroups of $G$.
\begin{lemma}\label{Lemma:fisdp}
Let $G\leq\Aut(T_d)$ be a group with the weak subgroup induction property. Suppose that if $H$ is a finite-index subgroup of $G$, then there exists a level $n$ such that $\Stab_H(\level n)$ is a subdirect product of $G^{d^n}$. Then $G$ has the subgroup induction property.
\end{lemma}
\begin{proof}
It follows from Definition~\ref{Definition2} and Definition~\ref{Definition:StrongSubgroupInduction} that it is enough to show that any inductive class contains all finite-index subgroups of $G$, but this follows from the hypothesis on finite-index subgroups of $G$ and the definition of an inductive class.
\end{proof}
%
%
%
%
%
\section{Subgroup induction property for \GGS{} groups}\label{Section:GGS}
Until now, only two groups were known to possess the subgroup induction property, namely the first Grigorchuk group acting on the binary rooted tree~\cite{MR2009443} and the Gupta-Sidki group acting on the ternary rooted tree~\cite{MR3513107}.
In this section, we exhibit an infinite family of groups with this property, the torsion \GGS{} groups.
\subsection{General facts about \GGS{} groups}
Defined as a generalisation of both the second Grigorchuk group and the Gupta-Sidki groups, the Grigorchuk-Gupta-Sidki groups, or \GGS{} groups, are a well-studied family of groups acting on rooted trees.
Let us first define them, before listing a few of their properties that will be useful later on.
\begin{definition}\label{def:GGSGroups}
Let $p$ be an odd prime number, let $T$ the $p$-regular tree and let $\mathbf{e}=(e_0,\dots,e_{p-2})$ be a vector in $(\mathbf{F}_p)^{p-1}\setminus\{0\}$.
The \GGS{} group $G_{\mathbf{e}}=\gen{a,b}$ with defining vector $\mathbf e$ is the subgroup of $\Aut(T)$ generated by the two automorphisms 
\begin{align*}
a&=\varepsilon\cdot(1,\dots,1)\\
b&=(a^{e_0},\dots,a^{e_{p-2}},b)
\end{align*}
where $\varepsilon$ is the cyclic permutation $(1 2 \dots p)$.
\end{definition}
\begin{remark}\label{Remark:Prime}
The definition of \GGS{} groups naturally extends to $n$-regular rooted trees for $n$ an integer greater than $1$.
For example, Vovkivsky proved Proposition~\ref{Prop:Torsion} in the context of $n=p^r$ a prime power.
However, it is usual to restrict the study of \GGS{} groups to the ones acting on $p$-regular rooted trees with $p$ prime, as some results (for example the classification of Lemma~\ref{lemma:ClassificationSubgroups}) depend on this fact.
This is why in the following we will always assume that $T$ is a $p$-regular rooted tree with $p$ prime.
\end{remark}
The \defi{Gupta-Sidki $p$-group} is the \GGS{} group acting with defining vector $(1,-1,0,\dots,0)$.

A \GGS{} group $G$ acting on the $p$-regular rooted tree $T$ is always infinite (since $\mathbf{e}\neq 0$,~\cite{MR1754681}), residually-(finite $p$).
By~\cite{MR3152720}, a \GGS{} group $G_{\mathbf{e}}$ is branch if and only if $\mathbf{e}$ is not constant, which is equivalent by~\cite{MR3652785} to the fact that $G_{\mathbf{e}}$ has the congruence subgroup property.
Even if the \GGS{} groups with constant defining vector are not branch, they are still weakly branch~\cite{MR3652785}.

If a \GGS{} group $G$ is torsion, then it is super strongly self-replicating~\cite{Uria}.
Furthermore, $G$ does not contain maximal subgroups of infinite index, a fact that is due to Pervova~\cite{MR2197824} in the torsion case and to the first-named author and Anitha Thillaisundaram in the non-torsion case~\cite{2020arXiv200502346F}.
This implies that every maximal subgroup of a branch \GGS{} group $G$ is normal and of index $p$, and thus, by Lemma~\ref{Lemma:DerivedSubgroup}, that the derived subgroup $G'$ is contained in $\Phi_F(G)=\Phi(G)$, the Frattini subgroup (that is, the intersection of all maximal subgroups) of~$G$.

In what follows, we will frequently make use, often without explicitly mentioning them, of several structural properties of \GGS{} groups, which we summarize in the proposition below.
A proof of these classical facts can be found, among other sources, in~\cite{MR3152720}.
\begin{proposition}\label{Prop:Derived}
Let $p$ be a prime number, let $G$ be a \GGS{} group acting on the $p$-regular rooted tree $T$, and let $G'$ be its derived subgroup.
Then,
\begin{enumerate}[(i)]
\item $\Stab_G(\level{1})=\gen{b}^G=\gen{b,aba^{-1}\dots,a^{p-1}ba^{-(p-1)}}$,
\item $G=\gen{a}\ltimes\Stab_G(\level{1})$,
\item $G/G'=\gen{aG',bG'}\cong C_p\times C_p$,\label{item:Abelianization}
\item $\Stab_G(\level{2})\leq G'\leq\Stab_G(\level{1})$.
\end{enumerate}
\end{proposition}
We will also use without mentioning it the following trivial fact.
\begin{lemma}
Let $T$ be a $d$-regular rooted tree and $a=\varepsilon\cdot(1,\dots,1)$ where $\varepsilon$ is the cyclic permutation $(12\dots d)$.
For any element $(x_1,x_2,\dots,x_d)$ of $\Stab(\level1)$ we have
\[(x_1,x_2,\dots,x_d)a=a(x_2,x_3,\dots,x_d,x_1). \]
\end{lemma} 
As it follows from Lemma~\ref{Lemma:Torsion} that non-torsion \GGS{} groups cannot have the subgroup induction property, we will now focus exclusively on torsion \GGS{} groups.
Luckily, there is a simple criterion, due to Vovkivksy, to decide whether a \GGS{} group is torsion or not.
\begin{proposition}[{\cite[Theorem 1]{MR1754681}}]\label{Prop:Torsion}
Let $p$ be a prime, let  $\mathbf{e}=(e_0,\dots,e_{p-2})$ be a vector in $(\mathbf{F}_p)^{p-1}\setminus\{0\}$ and let $G_{\mathbf{e}}$ be the corresponding \GGS{} group. Then, $G_{\mathbf{e}}$ is torsion if and only if
\[\sum_{i=0}^{p-2}e_i=0.\]
\end{proposition}
It follows from this criterion that torsion \GGS{} groups have non-constant defining vectors and hence are branch and have the congruence subgroup property.

Subgroups of torsion \GGS{} groups that do not fix the first level satisfy a dichotomy that will be crucial in establishing the subgroup induction property.
This dichotomy was first obtained by Garrido for Gupta-Sidki groups, and then generalised to all torsion \GGS{} groups by the second named author.
\begin{lemma}[\cite{MR3513107,L2019}]\label{Lemma:GGSSections}
Let $G$ be a torsion \GGS{} group and $H$ be a subgroup of $G$ that is not contained in $\Stab_G(\level{1})$.
Then either all first-level sections of $H$ are equal to $G$, or they are all contained in $\Stab_G(\level{1})$, so that $\Stab_H(\level{1})=\Stab_H(\level{2})$.
\end{lemma}
\subsection{Lengths in \GGS{} groups}
A \GGS{} group $G$ is always generated by two elements $a$ and $b$, as in Definition~\ref{def:GGSGroups}.
However, for our purposes, the word length corresponding to this generating set is not the most convenient, and we will prefer to it the word length given by the generating set consisting of all powers of $a$ and $b$.
To avoid all possible confusion, we will call the word length with this specific set of generators the \defi{total length}.
\begin{definition}\label{def:TotalLength}
Let $p$ be a prime number, let $G$ be a \GGS{} group and let $g\in G$ be any element.
The \defi{total length} of $g$, which we will denote by $\lambda(g)$, is the word length of $g$ with respect to the generating set $\{a,a^2, \dots, a^{p-1}, b, b^2, \dots, b^{p-1}\}$.
In other words, $\lambda(g)$ is the smallest $n\in \N$ such that
\[ g=s_1s_2\dots s_n\]
with $s_1,s_2,\dots, s_n \in \{a,a^2, \dots, a^{p-1}, b, b^2, \dots, b^{p-1}\}$.
\end{definition}
Since every \GGS{} group $G$ is generated by two elements $a$ and $b$, every element of $G$ can be expressed as an alternating product of powers of $a$ and powers of~$b$.
To simplify our analysis, we will often want to ignore the powers of $a$, since they have no action beyond the first level, and focus our attention only on the powers of $b$.
To this end, we introduce a ``length'' different from the total length, that we will call the \defi{$b$-length}.
\begin{definition}\label{def:BLength}
Let $p$ be a prime number, $G$ be a \GGS{} group and $g\in G$ be any element.
The \defi{$b$-length} of $g$, denoted by $\abs{g}$, is the smallest $n\in \N$ such that 
\[g=a^{i_1}b^{j_1}a^{i_2}b^{j_2}\dots a^{i_n}b^{j_n}a^{i_{n+1}}\]
with $i_1,\dots, i_{n+1}, j_1, \dots, j_n \in \Z$.
\end{definition}
\begin{remark}
Please note that despite what its name and notation might suggest, the $b$-length is not a norm in the usual sense, but merely a pseudo-norm, since non-trivial elements, namely the powers of $a$, can have $b$-length $0$.
\end{remark}
The total length and the $b$-length are related to each other, and there are relations between the length of an element and the length of its sections, as the next proposition shows.
These inequalities are all fairly standard, but we give brief proofs here for completeness.
\begin{proposition}\label{prop:LengthReductions}
Let $p$ be a prime number, $G$ be a \GGS{} group and $g\in G$ be any element of $G$, and let $v\in \level{1}$ be any vertex on the first level.
We have
\begin{enumerate}[(i)]
\item $2\abs{g}-1 \leq \lambda(g) \leq 2\abs{g}+1$,\label{item:RelBTotal}
\item $\lambda(\varphi_v(g))\leq \abs{g}$,\label{item:SectionTotalAndBLength}
\item $\lambda(\varphi_v(g))\leq \frac{\lambda(g)+1}{2}$,\label{item:SectionTotalLength}
\item $\abs{\varphi_v(g)} \leq \frac{\abs{g}+1}{2}$,\label{item:SectionBLength}
\item $\sum_{w\in \level{1}} \abs{\varphi_w(g)} \leq \abs{g}$.\label{item:SumOfSectionBLength}
\end{enumerate}
\end{proposition}
\begin{proof}
\noindent(\ref{item:RelBTotal}) Let us write $m=\lambda(g)$, and let $s_1,\dots, s_m\in \{a,\dots, a^{p-1}, b, \dots, b^{p-1}\}$ be such that $g=s_1\dots s_m$.
By the minimality of $m$, if $s_i$ is a power of $b$, then $s_{i+1}$, if it exists, must be a power of $a$.
Consequently, $g$ can be written as an alternating product of powers of $a$ and powers of $b$ with at most $\frac{m+1}{2}$ powers of~$b$.
The first inequality thus holds.

For the second inequality, let us write $n=\abs{g}$, and let $i_1,\dots, i_{n+1}, j_1, \dots j_n \in \Z$ be such that
\[g=a^{i_1}b^{j_1}a^{i_2}b^{j_2}\dots a^{i_n}b^{j_n}a^{i_{n+1}}.\]
As $g$ can be written as a product of at most $2n+1$ elements belonging to the set $\{a,\dots, a^{p-1}, b, \dots, b^{p-1}\}$, we must have $\lambda(g)\leq 2n+1$.

\noindent(\ref{item:SectionTotalAndBLength}) We will proceed by induction on $n=\abs{g}$.
If $n=0$, then $g=a^j$ for some $j\in \Z$.
In particular, for any $v\in \level{1}$, we have $\varphi_v(g)=1$, so the result holds in this case.
Now, let us assume that it holds for a given $n\in \N$.
Then, if $\abs{g}=n+1$, we have $g=g'b^ia^j$ for some $g'\in G$ with $\abs{g'}=n$ and some $i,j\in \Z$.
Since we have
\[\varphi_v(g) = \varphi_v(g'b^ia^j) = \varphi_{a^j\cdot v}(g')\varphi_{a^j\cdot v}(b^i)\]
for any $v\in \level{1}$, the result follows from the induction hypothesis, the fact that $\lambda(\varphi_w(b^{i}))\leq 1$ for all $w\in \level{1}$ and subadditivity of the total length.

\noindent(\ref{item:SectionTotalLength}) and (\ref{item:SectionBLength})  follow directly from (\ref{item:RelBTotal}) and (\ref{item:SectionTotalAndBLength}).

\noindent(\ref{item:SumOfSectionBLength}) This follows by induction on $n=\abs{g}$, using the fact that if $g=g_1g_2$, then
\[\sum_{w\in \level{1}}\abs{\varphi_w(g)} \leq \sum_{w\in\level{1}}\abs{\varphi_w(g_1)} + \sum_{w\in\level{1}}\abs{\varphi_w(g_2)},\]
and that we have $\sum_{w\in\level{1}}\abs{\varphi_w(a^i)}=0$ and $\sum_{w\in\level{1}}\abs{\varphi_w(b^i)}=1$ for all $i\in \{1,2,\dots, p-1\}$.
\end{proof}
\subsection{Proof of the subgroup induction property for torsion \GGS{} groups}
We will now prove that all torsion \GGS{} groups possess the subgroup induction property.
 
Let us begin by fixing some notation.
For the rest of this section, $p$ will denote an odd prime and $G$ will denote a torsion \GGS{} group with defining vector $\mathbf{e}=(e_0,\dots,e_{p-2})$ acting on the $p$-regular rooted tree $T$.
\begin{notation}
For any $g\in G$, we will denote by $\alpha_g$ and $\beta_g$ the unique elements of $\{0,1,\dots, p-1\}$ such that $g\equiv_{G'} a^{\alpha_g}b^{\beta_g}$.
Note that the existence of such elements follows from point (\ref{item:Abelianization}) of Proposition~\ref{Prop:Derived}.
\end{notation}
The subgroups of a torsion \GGS{} group $G$ can be classified into three broad classes, as the following lemma shows.
\begin{lemma}\label{lemma:ClassificationSubgroups}
Let $H\leq G$ be a subgroup.
Then, either
\begin{enumerate}[(I)]
\item $H=G$,\label{item:H=G}
\item there exists $j_H\in \{0,1\}$ such that $HG'/G' = \gen{b^{j_H}} G'/G'$,\label{item:HInStab}
\item there exists $j_H\in \{0,1,\dots, p-1\}$ such that $HG'/G'=\gen{ ab^{j_H}} G'/G'$.\label{item:HNotInStab}
\end{enumerate}
\end{lemma}
\begin{proof}
Since all maximal subgroups of $G$ are of finite index~\cite{MR2197824}, the Frattini subgroup $\Phi(G)$ is equal to $\Phi_F(G)$, see the discussion before Lemma~\ref{Lemma:DerivedSubgroup}.
On the other hand, since $G$ is a $p$-group, all its finite quotients are finite $p$-groups and hence nilpotent.
Therefore, by Lemma~\ref{Lemma:DerivedSubgroup} the derived subgroup $G'$ is contained in $\Phi_F(G)=\Phi(G)$.
This implies that $HG'$ is a proper subgroup of $G$ if and only if $H$ is proper.
Therefore, $HG'/G'$ is a proper subgroup of $G/G'$ if and only if $H$ is a proper subgroup of $G$.
As $G/G'$ is isomorphic to $C_p\times C_p$, its proper subgroups are all cyclic.
Thus, either $H=G$, or $HG'/G'$ is a cyclic subgroup of $G/G'$.
Cases II and III together cover all the possible cyclic subgroups of $G/G'$ and are mutually exclusive.
\end{proof}
In what follows, we will frequently need to study the first-level stabiliser of a given subgroup.
The following lemma gives us a convenient generating set for this stabilizer when the subgroup is of type~\ref{item:HNotInStab}.
\begin{lemma}\label{lemma:GenSetOfStab}
Let $H\leq G$ be a subgroup of type~\ref{item:HNotInStab} (according to the classification of Lemma~\ref{lemma:ClassificationSubgroups}) generated by a set $S\subseteq H$.
Then, for any $x\in S$ with $\alpha_x\ne 0$, the set
\[
	S'=\setst{x^{k_1}yx^{k_2}}{ y\in S, k_1,k_2\in \{0,1,\dots, p-1\}, (k_1+k_2)\alpha_x+\alpha_y \equiv 0 \mod p}
\]
generates $\Stab_H(\level1)$.
\end{lemma}
\begin{proof}
Let $x\in S$ be any element of $S$ such that $\alpha_x\ne 0$.
Then, since $H$ is of type~\ref{item:HNotInStab}, the set $\setst{x^{k}}{0\leq k \leq p-1}$ is a Schreier transversal for $\Stab_H(\level1)$.
Therefore, by the Reidemeister-Schreier method (see for example~\cite{MR2109550}), the set
\[
	S''=\setst{x^{k_1}yx^{-k_2}}{y\in S, k_1,k_2\in \{0,\dots, p-1\}, k_1\alpha_x+\alpha_y-k_2\alpha_x \equiv 0 \mod p}
\]
generates $\Stab_H(\level1)$.
Taking $k_1=p-1$ and $y=x$, which forces $k_2=0$, we see that $x^p\in S''$.
Thus, multiplying elements of $S''$ on the right by $x^p$ when $k_2$ is not zero, we can transform the generating set $S''$ into the set $S'$, which must consequently also generate $\Stab_H(\level1)$.
\end{proof}
The main idea behind the proof of the subgroup induction property for torsion \GGS{} groups is to use an argument of length reduction on the generators of finitely generated subgroups to reduce the question to subgroups generated by elements of length at most $1$.
We will establish this length reduction property in the next several lemmas.
While they are unfortunately rather technical at times, they ultimately make the proof of the main result very short and simple.
\begin{lemma}\label{lemma:LengthOfProjectionsOfPowers}
Let $x,y\in G$ be two elements, and suppose that $\alpha_x\ne 0$.
Let $k_1,k_2\in \{0,1,\dots, p-1\}$ be such that $(k_1+k_2)\alpha_x+\alpha_y \equiv 0 \mod p$.
Then, for every $v\in \level{1}$, there exist $l_1\in \{k_1, k_1-p\}$ and $l_2\in \{k_2, k_2-p\}$ such that
\[\abs{\varphi_v(x^{l_1}yx^{l_2})} \leq \frac{\abs{y}+1}{2} + \abs{x}.\]
Furthermore, this inequality is strict if $\abs{x}$ is odd.
\end{lemma}
\begin{proof}
It suffices to show that for every $k\in \{0,1,\dots, p-1\}$ and for every $w\in \level{1}$, we have $\abs{\varphi_w(x^{k})}+\abs{\varphi_w(x^{k-p})}\leq \abs{x}$.
Indeed, we have
\[\abs{\varphi_v(x^{l_1}yx^{l_2})}\leq \abs{\varphi_{w_1}(x^{l_1})}+\abs{\varphi_{w_2}(y)}+\abs{\varphi_{v}(x^{l_2})},\]
where $w_1=yx^{l_2}\cdot v$ and $w_2=x^{l_2}\cdot v$.
Therefore, if $\abs{\varphi_w(x^{k})}+\abs{\varphi_w(x^{k-p})}\leq \abs{x}$, one can choose $l_1\in \{k_1, k_1-p\}$ and $l_2\in \{k_2, k_2-p\}$ such that $\abs{\varphi_{w_1}(x^{l_1})}, \abs{\varphi_v(x^{l_2})}\leq \frac{\abs{x}}{2}$.
Note that if $\abs{x}$ is odd, we have $\abs{\varphi_{w_1}(x^{l_1})}, \abs{\varphi_v(x^{l_2})}< \frac{\abs{x}}{2}$, since the $b$-length of elements must always be an integer.
The result then immediately follows from the fact that $\abs{\varphi_{w_2}(y)}\leq \frac{\abs{y}+1}{2}$ for all $w_2\in \level{1}$ by Proposition~\ref{prop:LengthReductions} (\ref{item:SectionBLength}).

To conclude the proof, it thus remains only to show that
\[\abs{\varphi_w(x^{k})}+\abs{\varphi_w(x^{k-p})}\leq \abs{x}\] for all $k\in \{0,1,\dots, p-1\}$ and for all $w\in \level{1}$.
Let us write $x=a^{\alpha_x}(x_1, x_2, \dots, x_p)$.
By direct computation, we find that
\[\varphi_w(x^k) = x_{w+(k-1)\alpha_x} x_{w+(k-2)\alpha_x} \dots x_{w+\alpha_x}x_{w}\]
(with the exceptional case of $\varphi_w(x^0)=1$ and $\abs{\varphi_w(x^0)}=0$) and that
\[\varphi_w(x^{k-p})=x^{-1}_{w+k\alpha_x} x^{-1}_{w+(k+1)\alpha_x}\dots x^{-1}_{w+(p-1)\alpha_x}.\]
Therefore,
\begin{align*}
\abs{\varphi_w(x^k)} + \abs{\varphi_w(x^{k-p})} &\leq \sum_{i=0}^{k-1}\abs{x_{w+i\alpha_x}} + \sum_{i=k}^{p-1}\abs{x^{-1}_{w+i\alpha_x}} \\
& = \sum_{i=1}^{p} \abs{x_i} \\
&\leq \abs{x},
\end{align*}
where we used Proposition~\ref{prop:LengthReductions} (\ref{item:SumOfSectionBLength}) and the easily established fact that $\abs{g^{-1}}=\abs{g}$ for all $g\in G$.
This concludes the proof.
\end{proof}
\begin{lemma}\label{lemma:LengthReductionTypeII}
Let $H < G$ be a subgroup of type~\ref{item:HInStab} (according to the classification of Lemma~\ref{lemma:ClassificationSubgroups}) generated by a finite set $S$, and let $M$ be the maximal $b$-length of elements in $S$.
Then for all $v\in \level{1}$, the subgroup $\varphi_v(H)$ is generated by the set $S_v:=\varphi_v(S)$, whose elements are of $b$-length at most $\frac{M+1}{2}$.
In particular, if $M>1$, the elements of $S_v$ are all of length strictly less than $M$.
\end{lemma}
\begin{proof}
Since subgroups of type~\ref{item:HInStab} are contained in the stabiliser of the first level, $\varphi_v(H)$ is generated by $\varphi_v(S)$.
The inequality on the $b$-length is then given by Proposition~\ref{prop:LengthReductions} (\ref{item:SectionBLength}).
\end{proof}
\begin{lemma}\label{lemma:LenghtReductionTypeIII}
Let $H < G$ be a finitely generated subgroup of type~\ref{item:HNotInStab} (according to the classification of Lemma~\ref{lemma:ClassificationSubgroups}). If $H$ is generated by elements of $b$-length at most $M$, then for every $v\in \level{1}$, the subgroup $\varphi_v(\Stab_H(v))$ is generated by elements of $b$-length at most $\frac{3M}{2}$.
Furthermore, if there exists $x\in H$ with $\alpha_x\ne 0$ and $\abs{x}<M$, then $\varphi_v(\Stab_H(v))$ is generated by elements of $b$-length at most $\frac{3M-1}{2}$.
\end{lemma}
\begin{proof}
Since $H$ is generated by elements of $b$-length at most $M$, the set $S$ of all elements of $H$ of $b$-length at most $M$ is a finite generating set for $H$. 
Let $x\in S$ be an element such that $\alpha_x\ne 0$.
Such an element must necessarily exist, since $H$ is of type~\ref{item:HNotInStab}.
By Lemma~\ref{lemma:GenSetOfStab}, the set
\[
	S'=\setst{x^{k_1}yx^{k_2}}{y\in S, k_1,k_2\in \{0,1,\dots, p-1\}, (k_1+k_2)\alpha_x+\alpha_y \equiv 0 \mod p}
\]
generates $\Stab_H(\level1)$.
Notice that by taking $k_1=p-1$, $y=x$ and $k_2=0$, we find $x^p\in S'$.

Let us fix $v\in \level{1}$.
Since the action of $G/\Stab_G(\level 1)$ on $\level{1}$ is free, we have $\Stab_H(v) = \Stab_H(\level1)$, and is thus generated by $S'$.
Using Lemma~\ref{lemma:LengthOfProjectionsOfPowers}, we can define a set $S''$ obtained from $S'$ by replacing every element of the form $x^{k_1}yx^{k_2}\in S'$ with $y\ne x$ by an element $x^{l_1}yx^{l_2}$ such that
\[\abs{\varphi_v(x^{l_1}yx^{l_2})} \leq \frac{\abs{y}+1}{2} + \abs{x},\]
with $l_1\in \{k_1,k_1-p\}$ and $l_2\in \{k_2, k_2-p\}$.
This corresponds to multiplying on the left and on the right elements of $S'$ different from $x^p$ by either $x^p$ or the identity.
Therefore, $S''$ is also a generating set of $\Stab_H(v)$.

It follows that $\varphi_v(S'')$ is a generating set of $\varphi_v(\Stab_H(v))$. Let $M'$ be the maximal $b$-length of any element in this generating set.
By construction, we have
\[M'\leq \max\Bigl\{\frac{M+1}{2}+\abs{x}\Bigm\lvert\abs{\varphi_v(x^p)}\Bigr\}.\]

Let us first show that $\abs{\varphi_v(x^p)}\leq \abs{x}$.
As in the proof of Lemma~\ref{lemma:LengthOfProjectionsOfPowers}, writing $x=a^{\alpha_x}(x_1,x_2, \dots, x_p)$, we find by direct computation that
\[\varphi_v(x^p) = x_{v+(p-1)\alpha_x}x_{v+(p-2)\alpha_x}\dots x_{v+\alpha_x}x_v.\]
Therefore, by subadditivity and Proposition~\ref{prop:LengthReductions} (\ref{item:SumOfSectionBLength}), we have
\[\abs{\varphi_v(x^p)}\leq \sum_{i=0}^{p-1}\abs{x_{v+i\alpha_x}}=\sum_{i=0}^{p-1}\abs{x_i} \leq \abs{x}.\]

Thus, we have $M' \leq \max\bigl\{\frac{M+1}{2}+\abs{x}\bigm\lvert\abs{x} = \frac{M+1}{2}+\abs{x}\bigr\}$.
Furthermore, it follows from Lemma~\ref{lemma:LengthOfProjectionsOfPowers} that this inequality is strict if $\abs{x}$ is odd.

If $\abs{x}<M$, then we have $\abs{x}\leq M-1$, since the length must be an integer.
Therefore, $\varphi_v(\Stab_H(v))$ is generated by elements of $b$-length at most $\frac{3M-1}{2}$.
This proves the second part of the lemma.
To finish the proof of the first part, it remains only to treat the case where $\abs{x}=M$.

If $\abs{x}=M$ and $M$ is odd, then $\varphi_v(\Stab_H(v))$ is generated by elements of $b$-length strictly smaller than $\frac{3M+1}{2}$.
Again, since the length must be an integer, we can say that the generators of $\varphi_v(\Stab_H(v))$ are of $b$-length at most $\frac{3M-1}{2}$.

Similarly, if $\abs{x}=M$ and $M$ is even, we can bound the $b$-length of the generators of $\varphi_v(\Stab_H(v))$ by $\frac{3M}{2}$, since $\frac{3M+1}{2}$ is not an integer.

Thus, in all cases, $\varphi_v(\Stab_H(v))$ can be generated by elements of $b$-length at most $\frac{3M}{2}$, which concludes the proof.
\end{proof}
\begin{lemma}\label{lemma:LengthProjectionsStabiliser}
Let $H < G$ be a finitely generated subgroup of type~\ref{item:HNotInStab} (according to the classification of Lemma~\ref{lemma:ClassificationSubgroups}). If $H$ is generated by elements of $b$-length at most $M$, then for every $w\in \level{2}$, the subgroup $\varphi_w(\Stab_H(w))$ is generated by elements of $b$-length at most $\frac{3M+2}{4}$.
\end{lemma}
\begin{proof}
If $M=0$, then we must have either $H=\{1\}$ or $H=\gen{a}$.
Since $H$ is of type~\ref{item:HNotInStab}, the first case is impossible, and in the second case, we have $\varphi_w(\Stab_H(w))=\{1\}$ for all $w\in \level{2}$.
As the trivial subgroup is generated by elements of $b$-length at most $0$, the result holds in this case.
Therefore, for the rest of the proof, we may suppose that $M\geq 1$.

By Lemma~\ref{Lemma:GGSSections}, either $\varphi_v(\Stab_H(v))=G$ for every $v\in \level{1}$, or the subgroup $\varphi_v(\Stab_H(v))$ is of type~\ref{item:HInStab}, according to the classification of Lemma~\ref{lemma:ClassificationSubgroups}, for every $v\in \level{1}$.
In the first case, the result is true, since $\varphi_w(\Stab_H(w))=G$ for every $w\in \level{2}$ and is thus generated by elements of $b$-length at most $1 \leq \frac{3M+2}{4}$, using the fact that $M\geq 1$.
In the second case, by Lemma~\ref{lemma:LengthReductionTypeII}, it suffices to show that for every $v\in\level 1$, the subgroup $\varphi_v(\Stab_H(v))$ is generated by elements of $b$-length at most $\frac{3M}{2}$, which is the case according to Lemma~\ref{lemma:LenghtReductionTypeIII}.
\end{proof}
\begin{lemma}\label{lemma:Length1IsForever}
Let $H<G$ be a subgroup of $G$ generated by elements of $b$-length at most $1$. Then, for all vertices $v\in T$, the subgroup $\varphi_v(\Stab_H(v))$ is generated by elements of $b$-length at most $1$.
\end{lemma}
\begin{proof}
By induction, it suffices to prove the result for $v\in \level{1}$.

The proof depends on the type of $H$, according to the classification of Lemma~\ref{lemma:ClassificationSubgroups}.
If $H$ is of type~\ref{item:H=G}, or in other words, if $H=G$, then $\varphi_v(\Stab_H(v))=G$ for all $v\in \level{1}$, so the result is trivial in this case.
If $H$ is of type~\ref{item:HInStab}, the result immediately follows from Lemma~\ref{lemma:LengthReductionTypeII}.
Lastly, if $H$ is of type~\ref{item:HNotInStab}, then by Lemma~\ref{lemma:LenghtReductionTypeIII}, $\varphi_v(\Stab_H(v))$ is generated by elements of length at most $\frac{3}{2}$.
As the length must be an integer, the result follows.
\end{proof}
\begin{lemma}\label{lemma:CaseIIILength2}
Let $H<G$ be a subgroup of type~\ref{item:HNotInStab}, according to the classification of Lemma~\ref{lemma:ClassificationSubgroups}, generated by a finite set $S$. 
Let us assume that every element of $S$ is of $b$-length at most $2$ and of total length at most $3$.
Then $\varphi_w(\Stab_H(w))$ is generated by elements of $b$-length at most $1$ for all $w\in \level{2}$.
\end{lemma}
\begin{proof}
Using Lemma~\ref{Lemma:GGSSections} and similarly to the proof of Lemma~\ref{lemma:LengthProjectionsStabiliser}, we can suppose that $\varphi_v(\Stab_H(v))$ is of type~\ref{item:HInStab}, according to the classification of Lemma~\ref{lemma:ClassificationSubgroups}, for every $v\in \level{1}$.

Thanks to Lemma~\ref{lemma:LengthReductionTypeII}, it suffices to show that $\varphi_v(\Stab_H(v))$ is generated by elements of $b$-length at most $2$ for all $v\in \level{1}$, since the $b$-length must be an integer.

Let us first suppose that there exists a non-trivial $x\in S$ such that $\abs{x}<2$.
Then, we cannot have $x\in G'$, since non-trivial elements of $G'$ must have $b$-length at least $2$.
Indeed, any element of $b$-length strictly less than $2$ can be written as $a^{i}b^ja^k$ with $i,j,k\in \Z$, and for this element to be in $G'$, by Proposition~\ref{Prop:Derived} (\ref{item:Abelianization}), we must have $j\equiv 0 \mod p$ and $i\equiv-k \mod p$, which forces this element to be the identity.

Therefore, $x$ does not belong to $G'$, and as $H$ is a subgroup of type~\ref{item:HNotInStab}, we conclude that $\alpha_x\ne 0$.
It then follows from Lemma~\ref{lemma:LenghtReductionTypeIII} that $\varphi_v(\Stab_H(v))$ is generated by elements of $b$-length at most $\frac{5}{2}$, and thus at most $2$, for all $v\in \level{1}$.

Thus, it only remains to show that the result holds when every non-trivial element of $S$ is of $b$-length $2$.
Let us choose some element $x\in S$ such that $\alpha_x\ne 0$.
In light of Lemma~\ref{lemma:GenSetOfStab}, we need to show that for every $v\in \level{1}$, for every $y\in S$ and for every $k_1,k_2\in \{0,1,\dots, p-1\}$ such that $(k_1+k_2)\alpha_x + \alpha_y \equiv 0 \mod p$, we have
\[\abs{\varphi_v(x^{k_1}yx^{k_2})} \leq 2.\]
Since $x$ and $y$ are of $b$-length $2$ and of total length at most $3$, we must have $x=b^{m_x}a^{\alpha_x}b^{n_x}$ and $y=b^{m_y}a^{\alpha_y}b^{n_y}$, with $m_x,n_x,m_y,n_y \in \{1,2,\dots, p-1\}$.
Therefore, we have
\[x^{k_1}yx^{k_2} = b^*_0 b^*_{\alpha_x}b^*_{2\alpha_x}\dots b^*_{k_1 \alpha_x}b^*_{k_1\alpha_x + \alpha_y}b^*_{k_1\alpha_x + \alpha_y + \alpha_x} \dots b^*_{k_1\alpha_x+\alpha_y+k_2\alpha_x},\]
where the $*$ represent unspecified powers (possibly zero) and $b_l = a^lba^{-l}$.
For every $v\in \level{1}$, we have
\[\abs{\varphi_v(b^*_l)} =
\begin{cases}
1 & \text{ if } l\equiv v \mod p \\
0 & \text{ otherwise}.
\end{cases}\]
As $\varphi_v(x^{k_1}yx^{k_2}) = \varphi_{v}(b^*_{0})\varphi_{v}(b^*_{\alpha_x})\dots \varphi_{v}(b^*_{k_1\alpha_x + \alpha_y + k_2\alpha_x})$, it suffices to show that in the sequence
\[0,\alpha_x, 2\alpha_x, \dots ,k_1\alpha_x, k_1\alpha_x+\alpha_y, k_1\alpha_x+\alpha_y+\alpha_x, \dots, k_1\alpha_x+\alpha_y+k_2\alpha_x,\]
every number (modulo $p$) appears at most twice.
Since $\alpha_x\ne 0$, using the fact that $p$ is prime and that $k_1\in \{0,1,\dots, p-1\}$, we see that $0, \alpha_x, 2\alpha_x, \dots, k_1\alpha_x$ are all pairwise distinct, modulo $p$.
Likewise, $k_1\alpha_x+\alpha_y, k_1\alpha_x+\alpha_y+\alpha_x, \dots, k_1\alpha_x+\alpha_y+k_2\alpha_x$ are also pairwise distinct modulo $p$.
Consequently, a number (modulo $p$) can appear at most twice in the sequence
\[0,\alpha_x, 2\alpha_x, \dots ,k_1\alpha_x, k_1\alpha_x+\alpha_y, k_1\alpha_x+\alpha_y+\alpha_x, \dots, k_1\alpha_x+\alpha_y+k_2\alpha_x,\]
since if it appeared three times, it would have to appear at least twice in either $0, \alpha_x, 2\alpha_x, \dots, k_1\alpha_x$ or in $k_1\alpha_x+\alpha_y, k_1\alpha_x+\alpha_y+\alpha_x, \dots, k_1\alpha_x+\alpha_y+k_2\alpha_x$, by the pigeonhole principle, which is impossible.

Thus, $\varphi_v(\Stab_H(v))$ is generated by elements of $b$-length at most $2$ for all $v\in \level{1}$, which concludes the proof.
\end{proof}
\begin{lemma}\label{lemma:CanProjectToLength1}
Let $H\leq G$ be a finitely generated subgroup of $G$. There exists some $n\in \N$ such that for every $v\in \level{n}$, the subgroup $\varphi_v(\Stab_H(v))$ is generated by elements of $b$-length at most $1$.
\end{lemma}
\begin{proof}
Let us fix a finite set $S$ of generators of $H$, and let $M$ be the maximal $b$-length of elements in $S$.

If $H$ is a subgroup of type~\ref{item:H=G}, according to the classification of Lemma~\ref{lemma:ClassificationSubgroups}, then $H=G$.
In this case, the result is obviously true by taking $n=0$.
If $H$ is of type~\ref{item:HInStab}, then by Lemma~\ref{lemma:LengthReductionTypeII}, if $M>1$, the subgroups $\varphi_v(\Stab_H(v))$ are generated by elements of $b$-length strictly smaller than $M$ for every $v\in \level{1}$.
Lastly, if $H$ is of type~\ref{item:HNotInStab}, then by Lemma~\ref{lemma:LengthProjectionsStabiliser}, if $M>2$, the subgroups $\varphi_v(\Stab_H(v))=G$ are generated by elements of $b$-length strictly smaller than $M$ for every $v\in \level{2}$.

Thus, using induction and Lemma~\ref{lemma:Length1IsForever}, it suffices to prove the result for $H$ of type~\ref{item:HNotInStab} and generated by elements of $b$-length at most $2$.
Let $x\in S$ be a generator of $H$ such that $\alpha_x\ne 0$.
Since we have $\abs{x}\leq 2$, it follows from Lemma~\ref{lemma:LenghtReductionTypeIII} that $\varphi_v(\Stab_H(v))$ is generated by elements of $b$-length at most $3$ for all $v\in \level{1}$.
If $\varphi_v(\Stab_H(v))=G$ for all $v\in \level{1}$ we are done.
Otherwise, by Lemma~\ref{Lemma:GGSSections}, Lemma~\ref{lemma:LengthReductionTypeII} and Proposition~\ref{prop:LengthReductions} (\ref{item:SectionTotalAndBLength}), for every $w\in \level{2}$, the subgroup $\varphi_w(\Stab_H(w))$ is generated by elements of $b$-length at most $2$ and of total length at most $3$.
If $\varphi_w(\Stab_H(w))$ is of type~\ref{item:H=G} or~\ref{item:HInStab}, then the result follows from the argument above, and if it is of type~\ref{item:HNotInStab}, the result follows from Lemma~\ref{lemma:CaseIIILength2}.
\end{proof}
We are now missing only one piece to prove that a torsion \GGS{} group has the subgroup induction property, namely that a subgroup generated by elements of length at most $1$ must belong to every inductive class.
We prove this fact in the following two lemmas.
\begin{lemma}\label{lemma:ProjectionsInXMeansHInX}
Let $H\leq G$ be a subgroup of a torsion \GGS{} group $G$, and let $\mathcal{X}$ be an inductive class of subgroups of $G$.
If there exists $n\in \N$ such that $\varphi_v(\Stab_H(v))\in \mathcal{X}$ for all $v\in \level{n}$, then $H\in \mathcal{X}$.
\end{lemma}
\begin{proof}
We proceed by induction on $n$.
For $n=0$, the result is trivially true.
Let us suppose that the result is true for some $n\in \N$, and let $H\leq G$ be a subgroup such that $\varphi_w(\Stab_H(w))\in \mathcal{X}$ for all $w\in \level{n+1}$.
For each $v\in \level{n}$, let us write $H_v=\varphi_v(\Stab_H(v))$.
Our assumptions on $H$ imply that for every $v\in \level{n}$ and for every $u\in \level{1}$, we must have $\varphi_u(\Stab_{H_v}(u)) \in \mathcal{X}$.
Using Property~\ref{Item:DefSubgroupInduction3} of Definition~\ref{Definition:StrongSubgroupInduction} and the fact that $\Stab_{H_v}(\level1)=\Stab_{H_v}(u)$, we conclude that $\Stab_{H_v}(\level1)\in \mathcal{X}$.
Since $\Stab_{H_v}(\level1)$ is of finite index in $H_v$, we must also have $H_v\in \mathcal{X}$ by Property~\ref{Item:DefSubgroupInduction2} of Definition~\ref{Definition:StrongSubgroupInduction}.

We have just shown that for every $v\in \level{n}$, we have $\varphi_v(\Stab_H(v))\in \mathcal{X}$.
Thus, by our induction hypothesis, we have $H\in \mathcal{X}$.
\end{proof}
\begin{lemma}\label{lemma:Length1InX}
Let $\mathcal{X}$ be an inductive class, and let $H\leq G$ be a subgroup generated by elements of $b$-length at most $1$. Then, $H\in \mathcal{X}$.
\end{lemma}
\begin{proof}
We begin by observing that if $H$ is generated by elements of total length~$1$, then $H\in \mathcal{X}$.
Indeed, in this case, we have either $H=G$, $H=\gen{a}$ or $H=\gen{b}$. Since both $\gen{a}$ and $\gen{b}$ are finite subgroups, the result follows.

If $H$ is not generated by elements of total length $1$, then $H$ must be of type~\ref{item:HInStab} or~\ref{item:HNotInStab}, according to the classification of Lemma~\ref{lemma:ClassificationSubgroups}.
If $H$ is of type~\ref{item:HInStab}, then for every $v\in \level{1}$, the subgroup $\varphi_v(\Stab_H(v))$ is generated by elements of total length $1$ by Proposition~\ref{prop:LengthReductions} (\ref{item:SectionTotalAndBLength}), and thus belongs to $\mathcal{X}$ by the previous argument.
It follows that $H\in \mathcal{X}$.

Lastly, if $H$ is of type~\ref{item:HNotInStab}, then by Lemma~\ref{lemma:LenghtReductionTypeIII}, $\varphi_v(\Stab_H(v))$ is generated by elements of $b$-length at most $\frac{3}{2}$, and thus at most $1$, for every $v\in \level{1}$.
As $\varphi_v(\Stab_H(v))$ is either $G$ or of type~\ref{item:HInStab} by Lemma~\ref{Lemma:GGSSections}, the preceding argument shows that $\varphi_v(\Stab_H(v))\in \mathcal{X}$ for every $v\in \level{1}$.
Therefore, we must have $H\in \mathcal{X}$ by Lemma~\ref{lemma:ProjectionsInXMeansHInX}.
\end{proof}
We are now finally ready to prove Theorem~\ref{Thm:Intro3}.
\begin{theorem}\label{thm:GGSHaveSubgroupInduction}
Every torsion \GGS{} group has the subgroup induction property.
\end{theorem}
\begin{proof}
The proof follows the same general strategy as in~\cite{MR3513107}.

Let $G$ be a torsion \GGS{} group, and let $\mathcal{X}$ be an inductive class of subgroups.
We need to show that every finitely generated subgroup of $G$ belongs to $\mathcal{X}$.

Let $H\leq G$ be a finitely generated subgroup.
By Lemma~\ref{lemma:CanProjectToLength1}, there exists $n\in \N$ such that $\varphi_v(\Stab_H(v))$ is generated by elements of $b$-length at most $1$ for all $v\in \level{n}$.
Lemma~\ref{lemma:Length1InX} then implies that $\varphi_v(\Stab_H(v))\in \mathcal{X}$ for all $v\in \level{n}$.
Therefore, $H\in \mathcal{X}$ by Lemma~\ref{lemma:ProjectionsInXMeansHInX}.
\end{proof}
As a corollary, we get that all torsion \GGS{} groups are subgroup separable.
\begin{corollary}
Every torsion \GGS{} group is subgroup separable.
\end{corollary}
\begin{proof}
This is a direct application of Theorems~\ref{Thm:Intro3} and~\ref{Thm:IntroStrong}.
More precisely, in order to apply Theorem~\ref{Thm:IntroStrong}, we use the fact that torsion \GGS{} groups are finitely generated, self-similar, branch, and have the subgroup induction property.
\end{proof}

There are other consequences of the subgroup induction property that could be of interest for the study of torsion \GGS{} groups.
For instance, the structure of their finitely generated subgroups can be described using the notion of blocks, and one can classify their weakly maximal subgroups.
We refer the interested reader to~\cite{GLN2019,L2019,FGLN2023} for more information about these consequences.

\end{document}